\documentclass[10pt]{article}
\usepackage[a4paper]{geometry}
\usepackage{amsmath, amssymb, amsthm}

\usepackage{fullpage} 

\usepackage[colorlinks]{hyperref}
\hypersetup{linkcolor=black,citecolor=black,filecolor=black,urlcolor=blue}

\usepackage{appendix}
\usepackage{multicol} 
\usepackage{makeidx}
\usepackage{cite}
\usepackage{captdef}
\makeindex
\usepackage{paralist}
\usepackage{enumitem}

\usepackage{psfrag}
\usepackage{graphicx}
\usepackage{graphics}
\usepackage{color}

\usepackage{xcolor}

\usepackage[hypcap=false]{caption}
\usepackage{subcaption}

\newcommand{\R}{\mathbb{R}}

\newtheorem{theo}{Theorem}[section]

\newtheorem{lemma}[theo]{Lemma}
\newtheorem{prop}[theo]{Proposition}
\theoremstyle{definition}

\newtheorem{remark}[theo]{Remark}

\newcommand{\mF}{{\mathfrak F}}

\newcommand{\cI}{{\mathcal I}}

\newcommand{\cM}{{\mathcal M}}

\newcommand{\sign}{\operatorname{sgn}}

\newcommand{\eps}{\varepsilon}

\newcommand{\N}{{\mathbb{N}}}

\renewcommand{\epsilon}{\varepsilon}

\newcommand{\Sn}{{\mathbb S}^{N-1}}
\renewcommand{\leq}{\leqslant}
\renewcommand{\geq}{\geqslant}

\usepackage[normalem]{ulem}

\numberwithin{equation}{section}

\def\sideremark#1{\ifvmode\leavevmode\fi\vadjust{\vbox to0pt{\vss
 \hbox to 0pt{\hskip\hsize\hskip1em
 \vbox{\hsize2.1cm\tiny\raggedright\pretolerance10000
  \noindent #1\hfill}\hss}\vbox to15pt{\vfil}\vss}}}%

\title{On the least-energy solutions of the pure Neumann Lane-Emden equation}
\author{Alberto Salda\~{n}a,\footnote{Instituto de Matem\'aticas, Universidad Aut\'onoma de M\'exico, Circuito Exterior, Ciudad Universitaria, 04510 Coyoac\'an, Ciudad de M\'exico, M\'exico. \textit{Email address:}\ \texttt{alberto.saldana@im.unam.mx}}\ \ \& Hugo Tavares\footnote{CAMGSD and Mathematics Department, Instituto Superior T\'ecnico, Universidade de Lisboa, Av.\  Rovisco Pais, 1049-001 Lisboa, Portugal. \textit{Email address:}\ \texttt{hugo.n.tavares@tecnico.ulisboa.pt}}}

\date{}

\begin{document}

\maketitle

\begin{abstract}
We study the pure Neumann Lane-Emden problem in a bounded domain
\[
 -\Delta u = |u|^{p-1} u  \text{ in }\Omega, \qquad 
\partial_\nu u=0  \text{ on }\partial \Omega,
\]
in the subcritical, critical, and supercritical regimes. We show existence and convergence of least-energy (nodal) solutions (l.e.n.s.). In particular, we prove that l.e.n.s. converge to a l.e.n.s. of a problem with sign nonlinearity as $p\searrow 0$; to a l.e.n.s. of the critical problem as $p\nearrow 2^*$ (in particular, pure Neumann problems exhibit no blowup phenomena at the critical Sobolev exponent $2^*$); and we show that the limit as $p\to 1$ depends on the domain.  Our proofs rely on different variational characterizations of solutions including a dual approach and a nonlinear eigenvalue problem.  Finally, we also provide a qualitative analysis of l.e.n.s., including symmetry, symmetry-breaking, and monotonicity results for radial solutions.

\medbreak


\end{abstract}

\section{Introduction}

Let $\Omega\subset \R^{N}$, $N\geq 4$, be a bounded smooth domain and let $p\in [0,2^*-1]\backslash\{1\}$, where $2^*=\frac{2N}{N-2}$ is the Sobolev critical exponent. Consider the pure Neumann problem
\begin{equation}\label{eq:LENSNeumann}
 -\Delta u_p = |u_p|^{p-1} u_p  \text{ in }\Omega, \qquad 
\partial_\nu u_p=0  \text{ on }\partial \Omega,
\end{equation}
where, if $p=0$, equation~\eqref{eq:LENSNeumann} is understood as
\begin{align}\label{P0:intro}
 -\Delta u_0 = \sign(u_0) \ \text{ in }\Omega, \qquad 
\partial_\nu u_0=0  \text{ on }\partial \Omega,
\end{align}
and
\begin{align*}
 \sign(u)
 =\begin{cases}
      1, &\text{ if }u>0,\\
      -1, &\text{ if }u<0,\\
      0, &\text{ if }u=0.
  \end{cases}
\end{align*}

A least-energy solution of~\eqref{eq:LENSNeumann} is defined as a solution achieving the level
\begin{align}\label{c:eps}
L_{p}:=\inf \{I_p(w):\ w\in H^1(\Omega)\setminus \{0\} \text{ is a weak solution of~\eqref{eq:LENSNeumann}}\},
\end{align}
where $I_p:H^1(\Omega)\to \R$ is defined by
\begin{align}\label{I:def}
 I_p(u)=\frac{1}{2}\int_\Omega |\nabla u|^2- \frac{1}{{p}+1}\int_\Omega |u|^{{p}+1}.
\end{align}
Observe that a nontrivial classical solution $u_p$ of \eqref{eq:LENSNeumann} is always sign-changing, since the Neumann b.c. and the divergence theorem imply that
\begin{align}\label{comp}
\int_\Omega |u_p|^{p-1}u_p=0.
\end{align}
Then, $L_p$ is actually a least-energy \emph{nodal} level. 

The nature of the problem \eqref{eq:LENSNeumann} changes substantially in the following six cases: the sign nonlinearity $p=0,$ the sublinear regime $p\in(0,1)$, the eigenvalue problem $p=1$, the superlinear-subcritical regime $p\in(1,2^*-1)$, the critical exponent $p=2^*-1$, and the supercritical regime $p>2^*-1$.  For $p\leq 2^*-1$,  the existence of the corresponding least-energy nodal solutions (l.e.n.s.) relies on different variational methods and characterizations depending on the regime, each one facing different obstacles such as lack of compactness, of differentiability, concentration phenomena, scaling invariance, compatibility conditions, etc., see Proposition \ref{thm:varcha} below and references \cite{PW15,ST17,CK90,CK91,SW03}. Clearly, in the linear case $p=1$, solutions of \eqref{eq:LENSNeumann} exist if and only if 1 is an eigenvalue of the Neumann Laplacian, and $L_1=0$ in such a case.

\medskip

In this paper, we are interested in the relationship between these methods and in the convergence and qualitative properties of l.e.n.s. $u_p$ of \eqref{eq:LENSNeumann} for $p\in[0,2^*-1]$.  To connect these different approaches, we use some auxiliary problems that are easier to study when varying the exponent $p$.  In particular, our next result introduces a nonlinear eigenvalue problem which captures the essence of l.e.n.s. for $p>0$.  We fix first some notation.  We use $\|\cdot\|_t=\|\cdot\|_{L^t(\Omega)}$ to denote the standard $L^t$-norm. Let $\mu_1=\mu_1(\Omega)>0$ denote the first nonzero eigenvalue of the Neumann Laplacian and let $\psi_{1}\in C^\infty(\Omega)$ denote an $L^2$-normalized eigenfunction associated to $\mu_1$ with zero-average, namely a solution of
\begin{align}\label{varphi:intro}
-\Delta \psi_{1} = \mu_1\psi_{1}\ \text{ in }\Omega, \qquad
\partial_\nu \psi_{1}=0\ \text{ on }\partial \Omega,\qquad
\int_\Omega \psi_{1}=0,\quad 
\text{ and } \quad \|\psi_{1}\|_2=1.
\end{align}

For $p\in (0,2^*-1]$, define
\begin{align}\label{Lambdap:def}
\Lambda_p:= \inf \left\{ \int_\Omega |\nabla u|^2:\ u\in H^1(\Omega),\ \|u\|_{p+1}=1,\ \int_\Omega |u|^{p-1}u=0\right\}.
\end{align}
Observe that $\Lambda_1=\mu_1$ and that the  condition $\int_\Omega |u|^{p-1}u=0$ (a ``nonlinear average'' for $p\neq 1$) avoids the translation invariance typical in Neumann problems.

\begin{theo}\label{th:NeumannLENS_0}
Let $\Omega \subset \R^N$, $N\geq 4$, be a smooth bounded domain. Then $\Lambda_p$ is achieved for every  $p\in(0,\frac{N+2}{N-2}]$, every minimizer $v_p$ satisfies
\[
-\Delta v_p=\Lambda_p |v_p|^{p-1}v_p \text{ in } \Omega,\quad \partial_\nu v_p=0 \text{ on } \partial \Omega,
\]
and the map
\[
\left(0,\frac{N+2}{N-2}\right]\to \R^+;\qquad p\mapsto \Lambda_p
\]
is continuous.  Let $p_n,p \subset (0,\frac{N+2}{N-2}]$ be such that $p_n\to p$ as $n\to\infty$ and let $v_{p_n}$ be a minimizer for $\Lambda_{p_n}$. Then, up to a subsequence, there exists a minimizer $v_p$ of $\Lambda_p$ such that
\[
v_{p_n}\to v_{p} \qquad \text{ in }  C^{2,\alpha}(\overline \Omega)\ \forall \alpha\in (0,1) \quad \text{ as }n\to \infty.
\]
Moreover, 
\[
  \left(
  \frac{\Lambda_1}{\Lambda_{p_n}}
  \right)^\frac{1}{1-p_n}
  =e^{-\frac{1}{2}\int_\Omega
 v_{1}^2\ln(v_{1}^2)+o(1)} \qquad \text{ as } p_n\to 1.
 \]
 \end{theo}
 That $\Lambda_p$ is achieved is particularly delicate for $p=2^*-1$.  This can be deduced using the approach in \cite{CK91}, which only holds for $N\geq 4$ due to a compactness issue (see estimate~\eqref{nlt4} below). For $N=3$ it is still an open question whether the critical problem~\eqref{eq:LENSNeumann} admits nontrivial solutions in a general domain $\Omega$ (see \cite[p. 1135]{CK91}).  
 When $\Omega$ is symmetric with respect to a plane, \eqref{eq:LENSNeumann} admits at least one nontrivial solution, while if $\Omega$ is a ball it admits infinitely many solutions \cite{CK90}. However, even in these cases, the existence of least-energy solutions is not known. This is the reason for the restriction $N\geq 4$ in the statement of Theorem~\ref{th:NeumannLENS_0}.

The relation between $\Lambda_p$ and $L_p$ is clarified in the following result. 
 \begin{prop}\label{prop:rel_Lp_Dp}
 Let $N\geq 4$. For $p\in(0,\frac{N+2}{N-2}]\setminus \{1\}$, we have
 \begin{align}\label{eq:rel_Lp_Dp}
 L_p=\frac{{p}-1}{2({p}+1)}\Lambda_p^{\frac{{p}+1}{{p}-1}}.
\end{align} 
Moreover, $v_p$ achieves $\Lambda_p$ if, and only if, $u_p=\Lambda_p^{\frac{1}{p-1}}v_p$ is a least energy solution of \eqref{eq:LENSNeumann}. In particular, the set of least-energy solutions of~\eqref{eq:LENSNeumann} is nonempty for every $p\in(0,\frac{N+2}{N-2}]\backslash \{1\}$.
\end{prop}

We remark that, while $p\mapsto \Lambda_p$ is continuous, the map $p\mapsto L_p$ has a singularity at $p=1$ if $\Lambda_1=\mu_1\neq 1$.

The next result restates Theorem \ref{th:NeumannLENS_0} in terms of solutions of \eqref{eq:LENSNeumann}, including also asymptotics to the case $p=0$ (which are shown with different techniques, see the discussion below).

\begin{theo}\label{th:NeumannLENS}
Let $\Omega \subset \R^N$, $N\geq 4$, be a smooth bounded domain. Then
the set of least-energy solutions of~\eqref{eq:LENSNeumann} is nonempty for every $p\in[0,\frac{N+2}{N-2}]\backslash \{1\}$ and the map
\[
\left[0,\frac{N+2}{N-2}\right]\backslash \{1\}\to \R^+;\qquad p\mapsto L_p
\] 
is continuous. Let $(p_n)\subset (0,\frac{N+2}{N-2})\backslash \{1\}$ be such that $p_n\to p\in [0,\frac{N+2}{N-2}]$ as $n\to\infty$ and let $u_{p_n}$ denote a corresponding least-energy nodal solution of~\eqref{eq:LENSNeumann}.
\begin{enumerate}
\item \label{i1} If $p=0$, then, up to a subsequence, there is a least-energy solution $u_0$ such that
\begin{equation}\label{eq:convergence_C1a}
u_{p_n} \to u_{0} \qquad \text{ in }  C^{1,\alpha}(\overline \Omega)\ \forall \alpha\in (0,1) \quad \text{ as }n\to \infty.
\end{equation}
 \item  \label{i2} If $p\not\in\{0,1\}$, then, up to a subsequence, there is a least-energy solution $u_p$ such that
\begin{equation}\label{eq:convergence_C2a}
u_{p_n} \to u_{p} \qquad \text{ in }  C^{2,\alpha}(\overline \Omega)\ \forall \alpha\in (0,1) \quad \text{ as }n\to \infty.
\end{equation}
\item \label{i3} If $p=1$ and:
\begin{itemize}
\item if $\mu_1(\Omega)>1$ then 
\[
\lim_{p_n\searrow 1}\|u_{p_n}\|_{L^\infty(\Omega)}{=\lim_{p_n\searrow 1}L_{p_n}}=\infty \quad \text{ and } \quad \lim_{p_n\nearrow 1}\|u_{p_n}\|_{L^\infty(\Omega)}{=\lim_{p_n\nearrow 1}L_{p_n}}=0;
\]  
\item if $\mu_1(\Omega)<1$ then 
\[
\lim_{p_n\searrow 1}\|u_{p_n}\|_{L^\infty(\Omega)}{=\lim_{p_n\searrow 1}L_{p_n}}=0\quad \text{ and } \quad \lim_{{p_n}\nearrow 1}\|u_{p_n}\|_{L^\infty(\Omega)}{=\lim_{p_n\nearrow 1}L_{p_n}}=\infty.
\] 
\end{itemize}
\item \label {i4} If $p=\mu_1(\Omega)=1$ then, up to a subsequence, there is an eigenfunction $\psi_{1}$ satisfying~\eqref{varphi:intro} such that
\begin{align}\label{u1}
u_{p_n}\to u_1:=e^{-\frac{1}{2}\int_\Omega\psi_{1}^2\ln(\psi_{1}^2)}\psi_{1}\quad \text{ in $C^{2,\alpha}(\Omega)$ as $n\to \infty$,} 
\end{align}
and the map $p\mapsto L_p$ is continuous at $p=1$. \newline
In particular, if $\mu_1(\Omega)=1$ is a simple eigenvalue, then the limit of $(u_{p_n})$ is unique.
\end{enumerate}
\end{theo}
For $p>0$, Theorem \ref{th:NeumannLENS} is a direct consequence of Theorem \ref{th:NeumannLENS_0} and Proposition \ref{prop:rel_Lp_Dp}. The variational characterization of least-energy solutions in the case $p=0$ is given in \cite{PW15} and we do a careful asymptotic analysis of the energy and of minimizers. For $p<2^*-1$, an inspection of the proofs shows that Theorem \ref{th:NeumannLENS} holds also for $N\geq 1$, where as usual $2^*=\infty$ if $N=1,2.$  In two dimensions, the simplicity of the second eigenvalue $\mu_1$ can be deduced from geometrical properties of $\Omega$, see \cite[Section 2]{BB99}.  A detailed analysis on the shape of Neumann eigenfunctions associated to $\mu_1$ can be found in \cite{GiraoWeth}.

\smallbreak

Parts 3 and 4 of Theorem \ref{th:NeumannLENS} are Neumann counterparts of similar results obtained in \cite[Theorem 4]{BBGV} for a \emph{superlinear} Dirichlet version of \eqref{eq:LENSNeumann}, namely
\begin{equation}\label{cc}
 -\Delta u = |u|^{p-1}u   \text{ in }\Omega, \qquad 
 u=0  \text{ on }\partial \Omega.
\end{equation}
 The method we use is different from the one in \cite{BBGV} and it has the advantage of giving a precise limit \eqref{u1}, allowing us to also characterize the limit arising from the sublinear regime $p_n\nearrow 1$, (see Remark \ref{bbgv:rmk} for a comparison between the results from \cite{BBGV} and ours).
We remark that such a characterization of the limit of l.e.n.s. $u_{p_n}$ as $p_n\nearrow 1$ does not seem to be known for the Dirichlet case and that the variational characterization of the sublinear Dirichlet l.e.n.s. is particularly delicate, since in general it depends  on the domain (for instance if $\Omega$ is a ball then l.e.n.s. are of mountain pass type, whereas in dumbbell domains this is not the case, see \cite{BMPTW20} for the details). 

Another difference between the Neumann and the Dirichlet problems occurs at the critical regime $p=2^*-1$. Indeed, if $\Omega$ is starshaped there are no nontrivial solutions of \eqref{cc} at $p=2^*-1$ (see for instance \cite[Proposition 1.47]{W96}) and solutions typically blowup as  $p\nearrow 2^*-1$, see \cite{R89,H91} for least energy solutions and \cite{BartschMichelettiPistoia} for least-energy \emph{nodal} solutions. Therefore Theorem~\ref{th:NeumannLENS} shows a difference between the Dirichlet and the Neumann problems regarding the asymptotic behavior of l.e.n.s. as $p\nearrow 2^*-1$. 

\smallbreak

We complement our existence and convergence results with some qualitative analysis of the solutions in radial domains. It is known in several situations that l.e.n.s. have a simpler shape than other solutions: they are foliated Schwarz symmetric (axially symmetric and monotone in the angle variable, see Section \ref{fss:ssec} for a precise definition), not being, however, radially symmetric. This is shown for $p\in (0,2^*-1)$ in \cite[Corollary 1.4]{ST17} (see also \cite[Theorem 1.1]{PW15}). In this paper we extend these results to the critical case $p=2^*-1$, concluding the following.
\begin{prop}\label{prop:partialsymmetry}
 Let $N\geq 4$, let $\Omega$ be a ball and 
 $p\in (0,\frac{N+2}{N-2})$, or let $\Omega$ be an annulus and $p\in(0,\frac{N+2}{N-2}]$. Then every $v_p$ minimizer for $\Lambda_p$ is a nonradial foliated Schwarz symmetric function.
In particular, for $p\in(0,\frac{N+2}{N-2}]\backslash\{1\}$, the same conclusion is true for every least-energy solution $u_{p}$ of \eqref{eq:LENSNeumann}.
\end{prop}
The proof relies on an extension of the arguments in \cite{ST17} to the variational setting proposed in \cite{CK91} to study critical Neumann problems. This approach uses a rearrangement-type transformation in a dual setting devised recently in \cite{ST17}, together with a careful study of the radially symmetric l.e.n.s. (see Theorem \ref{th:NeumannLENS:radial} below). We observe that for $p=0$ the full picture is, up to our knowledge, still incomplete. In \cite[Theorem 1.2]{PW15} it is shown that $u_0$ is foliated Schwarz symmetric if $\Omega$ is a ball or an annulus, being not radial if $\Omega$ is a ball (however, the symmetry breaking on the annulus is an open problem). 

\medskip

As a by-product of our approach, we can state the following results, which are a version of Theorem \ref{th:NeumannLENS} and Proposition \ref{prop:partialsymmetry}  in the radial setting, where some additional monotonicity information can be deduced. For simplicity, we only state the results regarding $L_p$, not $\Lambda_p$.

Let $\Omega$ be a ball or an annulus. A least-energy radial solution of~\eqref{eq:LENSNeumann} is defined as a radially symmetric solution achieving the level
\begin{align*}
L_{p,rad}:=\inf \{I_p(w):\ w\in H_{rad}^1(\Omega)\setminus \{0\} \text{ is a weak solution of~\eqref{eq:LENSNeumann}}\},
\end{align*}
where $H_{rad}^1(\Omega):=\{u\in H^1(\Omega) : u\text{ is radially symmetric}\}$.  Let $\mu_{1,rad}=\mu_{1,rad}(\Omega)>0$ denote the first eigenvalue associated to a non constant radial eigenfunction of the Neumann Laplacian in $\Omega$. 

\begin{theo}[Least-energy radial solutions]\label{th:NeumannLENS:radial}
Let $\Omega \subset \R^N$, $N\geq 1$, be either a ball or an annulus and let
\begin{align}\label{Irad}
 \cI_{rad}:=\begin{cases}
     (0,2^*-1)&\text{ if $\Omega$ is a ball},\\
     (0,\infty)&\text{ if $\Omega$ is an annulus}.
    \end{cases}
\end{align}
 Then
the set of least-energy radial solutions of~\eqref{eq:LENSNeumann} is nonempty for every $p\in \cI_{rad}\setminus\{1\}$, and the map
\[
\cI_{rad}\setminus\{1\}\to \R^+;\qquad p\mapsto L_{p,rad}
\] 
is continuous. Let $(p_n)\subset \cI_{rad}\setminus\{1\}$ be such that $p_n\to p\in \overline{\cI_{rad}}$ as $n\to\infty$ and let $u_{p_n}$ denote the corresponding least-energy radial solutions of~\eqref{eq:LENSNeumann}.
\begin{enumerate}
\item \label{i1r} If $p=0$, then, up to a subsequence, there is a least-energy radial solution $u_0$ such that
\begin{equation*}
u_{p_n} \to u_{0} \qquad \text{ in }  C^{1,\alpha}(\overline \Omega)\ \forall \alpha\in (0,1) \quad \text{ as }n\to \infty.
\end{equation*}
 \item  \label{i2r} If $p\not\in\{0,1\}$, then, up to a subsequence, there is a least-energy radial solution $u_p$ such that
\begin{equation*}
u_{p_n} \to u_{p} \qquad \text{ in }  C^{2,\alpha}(\overline \Omega)\ \forall \alpha\in (0,1) \quad \text{ as }n\to \infty.
\end{equation*}
\item \label{i3r} If $p=1$ and:
\begin{itemize}
\item if $\mu_1(\Omega)>1$ then 
\[
\lim_{p_n\searrow 1}\|u_{p_n}\|_{L^\infty(\Omega)}{=\lim_{p_n\searrow 1}L_{p_n,rad}}=\infty \quad \text{ and } \quad \lim_{p_n\nearrow 1}\|u_{p_n}\|_{L^\infty(\Omega)}{=\lim_{p_n\nearrow 1}L_{p_n,rad}}=0;
\]  
\item if $\mu_1(\Omega)<1$ then 
\[
\lim_{p_n\searrow 1}\|u_{p_n}\|_{L^\infty(\Omega)}{=\lim_{p_n\searrow 1}L_{p_n,rad}}=0\quad \text{ and } \quad \lim_{{p_n}\nearrow 1}\|u_{p_n}\|_{L^\infty(\Omega)}{=\lim_{p_n\nearrow 1}L_{p_n,rad}}=\infty.
\] 
\end{itemize}

\item \label {i4r} If $\mu_{1,rad}=p=1$ then, up to a subsequence, there is an $L^2$-normalized radial eigenfunction $\psi_{1,rad}$ associated to $\mu_{1,rad}$ such that
\begin{align}\label{u1rad}
u_{p_n}\to u_1:=e^{-\frac{1}{2}\int_\Omega\psi_{1,rad}^2\ln(\psi_{1,rad}^2)}\psi_{1,rad}\quad \text{ in $C^{2,\alpha}(\Omega)$ as $n\to \infty$,}
\end{align}
and the map $p\mapsto L_{p,rad}$ is continuous at $p=1$. \newline
If $\mu_{1,rad}=1$ is a simple eigenvalue, then the limit of $(u_{p_n})$ is unique.
\item \label{i6r} If $\Omega$ is a ball, then there are no radial solutions for~\eqref{eq:LENSNeumann} if $p\geq 2^*-1$.
\end{enumerate}
\end{theo}

Observe that, due to the change of sign, radial solutions of~\eqref{eq:LENSNeumann} enclose a radial \emph{Dirichlet} solution of the Lane-Emden equation in a subdomain of $\Omega$, and therefore, due to the Pohozaev identity, radial solutions of~\eqref{eq:LENSNeumann} \textit{cannot} exist for $p=2^*-1$. As a consequence, radial solutions (and in particular l.e.n.s.) exhibit a concentration phenomenon as $p\nearrow 2^*-1$, which can be studied in detail, see \cite{GST20,IS21}.

\medskip

In \cite[Corollary 1.4]{ST17} it is shown that, in balls and annuli under $p<2^*-1$, radial l.e.n.s. are monotone in the radial variable. The next results extends this monotonicity to the critical and supercritical case in annuli.
\begin{prop}\label{new:prop}
 Under the assumptions of Theorem \ref{th:NeumannLENS:radial}, $u_p$ is monotone in the radial variable for $p\in\cI_{rad}\backslash\{0,1\}$.
\end{prop}
The proof relies on a dual-energy-preserving transformation, therefore it cannot be used for the case $p=0$, which lacks a dual formulation because the sign nonlinearity is not invertible. Note however that the convergence result Theorem \ref{th:NeumannLENS:radial} guarantees the existence of at least one monotone radial l.e.n.s for $p=0$. 

\medskip

In the last result of this introduction, we present a  proposition that clarifies the relationship between $L_p$ and Nehari-type sets. The first two parts are known, see \cite{PW15}, and we include them for completeness.
\begin{prop}\label{thm:varcha}
Let  $N\geq 4$, $\Omega \subset \R^N$ be a smooth bounded domain. Then the following holds.
\begin{enumerate}
\item For $p=0$, 
\begin{align*}
 L_0=\inf_{u\in\mathcal{M}_0}I_0(u),
\end{align*}
where $\mathcal{M}_0$ is the Nehari-type set
\[
 \cM_0:=\left\{ 
 u\in H^1(\Omega)\backslash\{0\}\::\: 
 \Big|
 |\{u>0\}|-|\{u<0\}|
 \Big|\leq |\{u=0\}|
 \right\}.
\]
Moreover, any $u\in \cM_0$ such that $I_0(u)=L_0$ is a solution to \eqref{P0:intro}.
\item For $p\in (0,1)$, 
\[
L_p=\inf\{I_p(u):\ u\in H^1(\Omega)\setminus \{0\},\ \int_\Omega |u|^{p-1}u=0\}.
\]
Moreover, any $u\in H^1(\Omega)\setminus \{0\}$ such that $\int_\Omega |u|^{p-1} u=0$ and $I_p(u)=L_p$ is a solution to \eqref{eq:LENSNeumann}.
\item For $p\in (1,2^*-1]$,
\[
L_p= \inf_{w\in \mathcal{M}_p} I_p(w)= \mathop{\inf_{w\in H^1(\Omega)\setminus\{0\}}}_{\int_\Omega |w|^{{p}-1}w=0} \sup_{t>0} I_p(tw),
\]
where $\mathcal{M}$ is the Nehari-type set
\[
\mathcal{M}_p=\left\{w\in H^1(\Omega)\backslash \{0\}:\ \int_\Omega |w|^{{p}-1}w=0,\ I'_p(w)w=0  \right\}.
\]
Moreover, any $u\in \mathcal{M}_p$ such that $I_p(u)=L_p$ is a solution to \eqref{eq:LENSNeumann}.
\end{enumerate}
\end{prop}

\medbreak

We now briefly explain the different variational approaches involved in our proofs and why they are needed.

If $p=2^*-1$, the critical case, the only available proof of existence of least-energy solutions relies on the dual method, see \cite{CK91}, which allows to handle the nonlinear constraint $\int_\Omega |u|^{2^*-2}u=0$ in \eqref{Lambdap:def}.  For a linear constraint $\int_\Omega u=0$ one can use a direct approach, as in \cite{GiraoWeth}, where the authors show the existence of a minimizer of  
\begin{align}\label{GW:pb}
\inf \left\{ \int_\Omega |\nabla u|^2:\ u\in H^1(\Omega),\ \|u\|_{2^*}=1,\ \int_\Omega u=0\right\},\quad N\geq 3,
\end{align}
and qualitative properties. However, the techniques in \cite{GiraoWeth} strongly rely on the linearity of the constraint and cannot be directly extended to \eqref{Lambdap:def}, see Remark~\ref{GW:rmk}.

Therefore, to prove Theorem \ref{th:NeumannLENS_0} and the case $p>0$ in Theorem \ref{th:NeumannLENS} we also rely on the dual method.  To describe this framework, we introduce some notation. Let
\begin{align}\label{Xp}
X_p=\Big\{f\in L^{\frac{p+1}{p}}(\Omega): \ \int_\Omega f =0\Big\},
\end{align}
and let $K$ denote the inverse (Neumann) Laplace operator with zero average, that is, if $h\in X_p(\Omega)$, then $u := Kh\in W^{2,\frac{p+1}{p}}(\Omega)$ is the unique strong solution of $-\Delta u = h$ in $\Omega$ satisfying $\partial_\nu u=0$ on $\partial \Omega$ and $\int_\Omega u = 0$, see Lemma \ref{l:reg} below.  In this setting, the (dual) energy functional $\phi_p:X_p\to \R$ is given by
\begin{align*}
 \phi_p(f):=\int_\Omega \frac{p}{p+1}|f|^{\frac{p+1}{p}}-f\, K f\ dx,\qquad f\in X_p.
\end{align*}
Since \eqref{comp} holds for any nontrivial solution, we require a suitable translation of $K$. Let $K_p:X^\frac{p+1}{p}\to W^{2,\frac{p+1}{p}}(\Omega)$ be given by 
\begin{align}\label{Ks:def}
K_p h:=K h +c_p(Kh)\qquad \text{ for some }\quad c_p(Kh)\in\R\quad \text{ such that }\int_\Omega |K_p h|^{p-1}K_p h=0.  
\end{align}
Then, a critical point $f$ of $\phi_p$ solves the dual equation
$K_{p} f=|f|^{\frac{1}{p}-1}f$ in $\Omega$.

Now, consider the variational problem
\begin{align}\label{Dpintro}
D_p:=\sup\left\{\int_\Omega fK f:\ f\in X_p,\ \|f\|_{\frac{{p}+1}{{p}}}=1\right\},\qquad p>0.
\end{align}
Note that the dual method has translated the nonlinear constraint $\int_\Omega |u|^{p-1}u=0$ into the linear constraint $\int_\Omega f =0$ in \eqref{Xp}. 

In \cite{CK91}, the authors show that a maximizer $f_{p}$ of \eqref{Dpintro} is achieved for $p=2^*-1$ and that $u_{p} := D_{p}^{-\frac{p}{p-1}}K_{p}f_p$ is a solution of the critical problem \eqref{eq:LENSNeumann}. In \cite{CK91} it is not proved that this is a l.e.n.s., but we show that this is indeed the case in Lemma \ref{lemma:D} below, establishing also that $D_p$ is the inverse of $\Lambda_p$, see Lemma \ref{lemma:DpLambdap}.

On the other hand, one cannot consider $p=0$ in \eqref{Dpintro}. To study the asymptotics as $p\searrow 0$ we use a direct approach and the variational characterizations established in \cite{PW15} and that we recalled in Proposition \ref{thm:varcha}.

Finally, we do not know of any variational approach that can be used to show the existence of l.e.n.s. of \eqref{eq:LENSNeumann} in the supercritical regime $(p>2^*-1)$. If $\Omega$ is an annulus, then one can easily find a radially symmetric solution with least energy among radial functions, recall Theorem \ref{th:NeumannLENS:radial}. 

\medskip

The paper is organized as follows. In Section \ref{sec:reg} we show regularity estimates for solutions. Section \ref{sec:3} contains the proof of the existence and asymptotic behavior of l.e.n.s. for $p>0$. In particular, Subsection \ref{dm:sec} presents in detail the dual method, the study of $D_p$ the convergence properties of extremizers. Then Subsection \ref{sec:Lambda_p} transfers that information to the study of $\Lambda_p$ (where we prove Theorem \ref{th:NeumannLENS_0}) and in   Subsection \ref{sec:L_p} we apply those results to the least-energy level $L_p$, proving  Proposition \ref{prop:rel_Lp_Dp}, Theorem \ref{th:NeumannLENS}--2,3,4 and Proposition \ref{thm:varcha}.  In Section \ref{p0:sec} we analyze the convergence at $p=0$, proving Theorem \ref{th:NeumannLENS}--1. Section \ref{symm:sec} is devoted to the proof of our symmetry and symmetry-breaking results. The monotonicity of radial solutions is also shown here, see Subsection \ref{proofsradial:sec}.  Finally, in Section \ref{num:sec}, we present some numerical approximations of radial solutions in an annulus to illustrate their behavior when $p$ varies in the interval $[0,\infty)$.

\section{Regularity}\label{sec:reg}

Implicit in the statement of Theorem \ref{eq:LENSNeumann} and Theorem~\ref{th:NeumannLENS} is the fact that, for $p\in (0,2^*-1]\backslash \{1\}$, weak solutions of the problem~\eqref{eq:LENSNeumann}  (i.e. critical points of $I_p$) are of class $C^{2,\alpha}(\overline \Omega)$ for every $\alpha\in (0,1)$. This is a consequence of elliptic regularity together with either a bootstrap or a Brezis-Kato type argument, Sobolev inequalities, and Schauder estimates. Since these arguments are not that well known in the context of Neumann boundary value problems, we provide here the details for completeness, see Proposition~\ref{prop:regularity} below.  The starting point is  the following  regularity result for the Neumann problem~\eqref{eq:LENSNeumann} (see \cite[Theorem and Lemma in page 143]{RR85} and also \cite[Theorem 15.2]{ADN59}).
\begin{lemma}\label{l:reg}
Let $N\geq 1$, $\Omega$ be a smooth bounded domain in $\R^N$, $t>1$,  and $h\in L^t(\Omega)$ with $\int_\Omega h =0$. Then there is a unique strong solution $u\in W^{2,t}(\Omega)$ of 
 \begin{align}\label{Nprob}
  -\Delta u = h \text{ in }\Omega, \qquad \partial_\nu u=0 \text{ on }\partial \Omega,\qquad \int_\Omega u = 0,
  \end{align}
 and there is $C=C(\Omega,t)>0$ such that $\|u\|_{W^{2,t}}\leq C\|h\|_t.$
\end{lemma}

The following lemma yields that weak solutions of \eqref{eq:LENSNeumann} are also classical ones for $p\in(1,2^*-1]$, being also a key point in the proof of the convergence in~\eqref{eq:convergence_C2a}.

\begin{lemma}\label{lemma:regularitysol} Let $N\geq 3$, $\Omega$ be a smooth bounded domain in $\R^N$ and let $h,  H\in L^\frac{N}{2}(\Omega)$ such that $|h|\leq H$ in $\Omega$. If $u\in H^1(\Omega)$ is a weak solution of the problem
\[
-\Delta u=h(x) u \text{ in } \Omega,\qquad \partial_\nu u=0 \text{ on }  \partial \Omega,
\]
then  $u\in L^t(\Omega)$ for every $t\geq 1$ and there exists   $C=C(H,\Omega,t,N)>0$ such that $\|u\|_{t}\leq C \|u\|_{2^*}$.
\end{lemma}
\begin{proof} We use the standard procedure introduced by Brezis and Kato \cite{BK79}, adapting the proof of \cite[Lemma B.2]{Struwebook} to a Neumann problem. 
\smallbreak
\noindent  1)  Since $u\in H^1(\Omega)$, the function belongs in particular to $L^{2^*}(\Omega)$.  Assume in general that $u\in L^{2s+2}(\Omega)$ for some $s> 0$. Given $L>0$, consider the $L^\infty$-function $v=v_L:=\min\{|u|^s,L\}$. Observe that $uv^2,uv\in H^1(\Omega)$, with
\begin{align*}
\int_\Omega \nabla u \nabla (uv^2)&=\int_\Omega v^2|\nabla u|^2+2s\int_{\{|u|^s<L\}}|u|^{2s}|\nabla u|^2,\\
	 \int_\Omega |\nabla (uv)|^2&=\int_\Omega v^2|\nabla u|^2+(2s+s^2)\int_\Omega |\nabla u|^2.
 \end{align*}
Then, using H\"{o}lder's inequality and $uv^2$ as a test function for the equation $-\Delta u=h(x) u$ we have, for any $\eta>0$, 
\begin{align*}
\int_\Omega |\nabla (uv) |^2 &\leq C(s) \int_\Omega \nabla u \nabla (uv^2)\leq C(s)\int_\Omega |h| u^2v^2\leq C(s)\int_\Omega H u^2v^2\\
&\leq C(s)(\int_{\{H< \eta\}} \eta|u|^{2s+2}+\int_{\{H\geq \eta\}} H u^2v^2)\\
&\leq  C(s) ( \eta \|u\|_{2s+2}^{2s+2} + \eps(\eta) \|uv\|_{\frac{2N}{N-2}}^2 ),
\end{align*}
where $\eps(\eta):=\left(\int_{\{H\geq \eta\}}H^\frac{N}{2}\right)^\frac{2}{N}$ and $C(s)>0$ is a constant depending only on $s$. This, combined with Sobolev's inequality, Poincar\'e-Wirtinger's  inequality ($\| w-\frac{1}{|\Omega|}\int_\Omega w\|_2\leq C_P\|\nabla w\|_2$ for every $w\in H^1(\Omega)$), and the fact that $|v|\leq |u|^s$ in $\Omega$, yields
\begin{align*}
\|uv\|_{\frac{2N}{N-2}}^2 &\leq 2\left\|uv-\frac{1}{|\Omega|}\int_\Omega uv\right\|_{\frac{2N}{N-2}}^2 + 2|\Omega|^{-\frac{N+2}{N}}\| uv\|_1^2\\
				&\leq 2S_N^2\left\|uv-\frac{1}{|\Omega|}\int_\Omega uv\right\|_{H^1}^2 + 2|\Omega|^{-\frac{N+2}{N}}\| uv\|_1^2 \\
				&\leq   2S_N^2(1+C_P^2)\int_\Omega |\nabla (uv) |^2+ 2|\Omega|^{-\frac{N+2}{N}}\| u\|_{1+s}^{2+2s}\\
				&\leq  (2 S_N^2 (1+C_P^2)C(s) \eta+ 2|\Omega|^{-2/N}) \|u\|_{2+2s}^{2+2s} +2S_N^2(1+C_P^2) C(s) \eps(\eta) \|uv\|_{\frac{2N}{N-2}}^2.
\end{align*}
If we choose $\eta>0$ sufficiently large so that $2S_N^2 (1+C_P^2) C(s) \eps(\eta)=\frac{1}{2}$, we get the existence of $\kappa=\kappa(H,\Omega,N,s)>0$ such that
\[
\|uv\|_{\frac{2N}{N-2}}^2 \leq \kappa \| u\|_{2+2s}^{2+2s}.
\]In particular, since $\kappa$ is independent of $L>0$, we have 
\begin{equation}\label{eq:iteration_ineq}
\|u\|_{\frac{(1+s)2N}{N-2}}\leq \kappa^\frac{1}{2+2s}\|u\|_{2+2s}.
\end{equation}
\smallbreak
\noindent 2) Consider the sequence defined by
\[
s_0:=\frac{2}{N-2},\qquad 2+2s_{k+1}:=\frac{(1+s_k)2N}{N-2} \iff s_{k+1}=\frac{N s_k+2}{N-2},\ k\in \N_0
\]
which is increasing and diverges to infinity. From the previous step we conclude that $u\in L^{\frac{(1+s_k)2N}{N-2}}(\Omega)$ for every $k\in \N_0$ and therefore $u\in L^t(\Omega)$ for every $t\geq 1$. The final estimate is a consequence of the iteration procedure together with~\eqref{eq:iteration_ineq}.
\end{proof}

\begin{prop}\label{prop:regularity}
Let $p\in (0,2^*-1]$, $\lambda>0$ and let $u$ be a weak solution of the problem
\[
-\Delta u=\lambda |u|^{p-1}u \text{ in } \Omega,\quad \partial_\nu u=0 \text{ on } \partial \Omega.
\]
Then $u\in C^{2,\alpha}(\overline \Omega)$ for every $\alpha\in (0,1)$.
\end{prop}
\begin{proof}
If $0<p<2^*-1$, then a classical bootstrap method (using Lemma~\ref{l:reg} and Sobolev embeddings) shows that $u\in W^{2,t}(\Omega)$ for every $t\geq 1$. Thus, again by Sobolev embeddings, $u$ is H\"older continuous for any exponent $\alpha\in (0,1)$, and the Schauder estimate \cite[Theorem 6.31]{GT98} (see also the remark after the theorem) yields the desired regularity.

As for the case $p=2^*-1$, we have $-\Delta u=h(x) u$, for $h(x)=\lambda|u|^{\frac{4}{n-2}}\in L^\frac{N}{2}(\Omega)$. Then Lemma~\ref{lemma:regularitysol} implies that $u\in L^t(\Omega)$ for every $t\geq 1$, and Lemma~\ref{l:reg} shows that solutions are in $W^{2,t}(\Omega)$ for every $t\geq 1$, and we can conclude as before.
\end{proof}

\section{Existence and asymptotics for $p>0$}\label{sec:3}
The main purpose of this section is to prove Theorem \ref{th:NeumannLENS_0} and Theorem~\ref{th:NeumannLENS} (in the case $p>0$). Instead of using the functional $I_p$ defined in \eqref{I:def}, we rely on a dual method approach, combining  know facts for least-energy solutions proved in the subcritical case \cite{ST17} and in the critical one \cite{CK91} with a careful asymptotic analysis. We start by studying the associated level $D_p$ (recall \eqref{Dpintro}); transferring afterwards the results to $\Lambda_p$ (Theorem \ref{th:NeumannLENS_0}), which we relate with the least-energy level $L_p$ (Proposition \ref{prop:rel_Lp_Dp}). Then Theorem~\ref{th:NeumannLENS}, 2--4 and the equivalent characterizations of Proposition \ref{thm:varcha} follow as a straightforward result.

In this section we always assume that $N\geq 4$.

\subsection{The dual method}\label{dm:sec}
 
Following the notations in \cite{ST17}, we define 
\[
X_p = \left\{\, f\in L^\frac{{p}+1}{{p}}(\Omega)\::\: \int_\Omega f=0\, \right\}\qquad \text{for $p\in (0,2^*-1]$}.
\] 

In the following,
\begin{align*}
\text{$K:X_p\to W^{2,\frac{{p}+1}{{p}}}(\Omega)$ denotes the inverse zero-average Neumann Laplacian operator;}
\end{align*}
that is, for $h\in X_p$,  $u:=K h$ is the unique solution of $-\Delta u= h$ in $\Omega$, $u_\nu=0$ on $\partial \Omega$, and $\int_\Omega u=0$. 
 The operator $K$ is well defined and continuous by Lemma~\ref{l:reg}. Using the Rellich-Kondrachov compactness theorem and Lemma~\ref{l:reg}, we have that
\begin{align}
& \text{for $p<2^*-1$, the operator $K$ is compact from $L^\frac{{p}+1}{{p}}(\Omega)$ to $L^{{p}+1}(\Omega)$,} \label{K:com} \\
& \text{for $p=2^*-1$, $K$ is compact from $L^\frac{2N}{N+2}(\Omega)$ to $L^q(\Omega)$, for every $1\leq q<\frac{2N}{N-2}$.}\label{K:com2}
\end{align}

Let $\phi_p:X_p\to \R$ be the $C^1$-functional  given by
\begin{align}\label{phi:def}
\phi_p(f):=\frac{{p}}{{p}+1}\int_\Omega |f|^\frac{{p}+1}{{p}}-\frac{1}{2}\int_\Omega f K f,
\end{align}
whose derivative is
\[
\phi_p'(f)g=\int_\Omega |f|^{\frac{1}{{p}}-1}g f- \int g Kf\qquad \text{ for } g\in X_p.
\]Thus, at a critical point  $f$ of $\phi_p$, we have
\[
\phi_p(f)=\frac{{p}-1}{2({p}+1)}\int_\Omega fKf=\frac{{p}-1}{2({p}+1)}\int_\Omega |f|^{\frac{{p}+1}{{p}}}.
\]
For future reference, we observe that the operator
\begin{align}\label{eq:wc}
L^\frac{p+1}{p}(\Omega)\to \R;\  f\mapsto \int_\Omega fKf \qquad\text{ is weakly continuous for } p<2^*-1.
\end{align}
Following \cite{PW15, ST17},  define the map $c_p: L^p(\Omega)\to \R$ where, for each $w\in L^p(\Omega)$, $c_p(w)$ is the \textit{unique} real number such that $\int_\Omega |w|^{p-1}w=0$ (uniqueness follows from the strict monotonicity of the map $t\mapsto |t|^{{p}-1}t$ in $\R$). Then, as in \cite{ST17}, we define the (nonlinear) operator $K_p: X_p\to W^{2,\frac{{p}+1}{{p}}}(\Omega)$ by
\begin{align}\label{c:def}
K_p h= Kh+ c_p(Kh),
\end{align}
which by definition satisfies $\int_\Omega |K_p h|^{{p}-1}K_p h=0$.  Then $f$ is a critical point of $\phi_p$ if and only if $f$ is a weak solution of

\[
K_p f= |f|^{\frac{1}{{p}}-1}f\quad \text{ in }\Omega.
\]
(see \cite[Lemma 2.3]{ST17} for the details).

Inspired by \cite{CK91}, we use a dual version of $\Lambda_p$, which is of crucial importance in the symmetry results, see Section \ref{symm:sec} below. Let 
\begin{align}\label{D}
D_p:=\sup\left\{\int_\Omega fK f:\ f\in X_p,\ \|f\|_{\frac{{p}+1}{{p}}}=1\right\}\qquad \text{ for } p\in (0,2^*-1].
\end{align}

\begin{lemma}\label{lemma:bounds} 
Given $0<a<2^*-1$, there exists $\delta>0$ such that $D_p>\delta$ for every $p\in [a,2^*-1]$.
\end{lemma}
\begin{proof}
Let $0<a<2^*-1$ and let $\psi_1$ be an $L^{2}$-normalized eigenfunction with zero average associated to the first nonzero Neumann eigenvalue. Then $\zeta_1=\psi_1/\|\psi_1\|_{\frac{p+1}{p}}$ satisfies $-\Delta \zeta_1=\mu_1 \zeta_1$  in $\Omega$, $\|\zeta_1\|_\frac{p+1}{p}=1$, $\int_\Omega \zeta_1=0,$
and $0<\delta<(\mu_1 \|\psi_1\|^2_\frac{p+1}{p})^{-1} =\int_\Omega \zeta_1 K \zeta_1\leq D_p$ for $p\in[a,2^*-1]$.
\end{proof}

\begin{lemma}\label{lemma:CK91}
Let $p\in (0,2^*-1]$ and let $f_k$ be a maximizing sequence for $D_{p}$, that is,
\begin{align*}
f_k\in X_{p},\qquad \|f_k\|_{\frac{p+1}{p}}=1\quad \text{ for all } k\in\N,\qquad \int_\Omega f_k K f_k \to D_{p}\quad \text{ as }k\to\infty.
\end{align*}
There exists $f\in X_p$ such that (up to a subsequence) $f_k \to f$ \emph{strongly} in $L^\frac{p+1}{p}(\Omega)$ as $k\to\infty$.
\end{lemma}
\begin{proof} 
Let $f_k$ be as in the statement and let $p\in(0,2^*-1]$. We split the proof in two cases.

\textbf{Subcritical case $0<p<2^*-1$:}  
Then there is $f\in L^\frac{{p}+1}{{p}}(\Omega)$ such that (up to a subsequence) $f_k\rightharpoonup f$ weakly in $L^\frac{{p}+1}{{p}}(\Omega)$, and $K f_k \to K f$ strongly in $L^{{p}+1}(\Omega)$, as $k\to \infty$, by~\eqref{K:com}. Thus $\|f\|_{\frac{{p}+1}{{p}}}\leq 1$ and $\int_\Omega f K f=D_p>0$ (see \eqref{eq:wc}); hence $f\neq 0$, and
\[
\frac{1}{\|f\|^2_{\frac{{p}+1}{{p}}}}D_p \leq  \int_\Omega \frac{f}{\|f\|_{\frac{{p}+1}{{p}}}} K \left(\frac{f}{\|f\|_{\frac{{p}+1}{{p}}}}\right)  \leq D_p,
\]
so that $\|f\|_{\frac{{p}+1}{{p}}}=1$. Since $\frac{p+1}{p}>1$, then $L^\frac{p+1}{p}(\Omega)$ is a uniformly convex Banach space and so weak convergence together with convergence of the norms implies that $f_k$ converges strongly to $f$ (see for instance \cite[Proposition 3.32]{Brezis}).

\smallbreak

\textbf{Critical case $p=2^*-1$:} 

 Define 
\[
F(f)=\frac{1}{2}\int_\Omega fKf,\qquad G(f)=\frac{p}{p+1}\int_\Omega |f|^{\frac{p+1}{p}}.
\] 
By Ekeland's variational principle (see \cite[Theorem~1.1 and Theorem~3.1]{Ekeland}), there exists $\lambda_k\in \R$ and $g_k\in X_{p}$ with 
\begin{align*}
\|g_k\|_{\frac{p+1}{p}}=1      
\end{align*}
 such that (up to a subsequence)
\begin{align}\label{ekeland}
\|g_k-f_k\|_{\frac{p+1}{p}}\to 0,\qquad    F(g_k)\to \frac{1}{2}D_{p},\qquad F'(g_k)-\lambda_k G'(g_k)\to 0 \text{ in } X_{p}^*,
\end{align}
as $k\to\infty$.
 
Since $g_k$ is bounded in $L^\frac{p+1}{p}(\Omega)$, there is $g_0\in L^\frac{p+1}{p}(\Omega)$ such that (up to a subsequence) $g_k\rightharpoonup g_0$ weakly in  $L^\frac{p+1}{p}(\Omega)$  as $k\to\infty$. Our aim is to show that this convergence is, in fact, strong. Having this in mind, we start by claiming that
\begin{align}\label{ae}
g_k \to g_0\qquad \text{ a.e. in }\Omega  \text{ as $k\to\infty$.}
\end{align}
 By the boundedness of $g_k$, we have
\[
o(1)=F'(g_k)g_k-\lambda_k G'(g_k)g_k= \int_\Omega g_k K g_k -\lambda_k \int_\Omega |g_k|^{\frac{p+1}{p}}= D_{p} -\lambda_k  + o(1),
\]
therefore
\begin{align}\label{lambda:c}
\lambda_k\to D_{p}\neq 0. 
\end{align}
Moreover, by Lemma~\ref{lemma:regularitysol}, there is $C(\Omega,n)>0$ such that $\|Kg_k\|_{W^{2,\frac{p+1}{p}}(\Omega)}\leq C \|g_k\|_{\frac{p+1}{p}}=C$ and therefore, by~\eqref{K:com2}, there is $w\in L^\frac{p+1}{p}(\Omega)$ such that (up to a subsequence)
\begin{align}\label{K:con}
Kg_k\to w\quad \text{ in }L^q(\Omega)\text{ for any }1<q<\frac{p+1}{p} \text{ and pointwisely a.e. in $\Omega$} \text{ as }k\to\infty.
\end{align}

Let $\psi\in L^\frac{p+1}{p}(\Omega)$ and $\varphi:= \psi-\frac{1}{|\Omega|}\int_\Omega \psi$. By the triangle inequality and H\"older's inequality, we have that
\begin{align}\label{cota}
\| \varphi\|_\frac{p+1}{p} \leq \| \varphi\|_\frac{p+1}{p}+\Big(|\Omega|^{-1}\int_\Omega \psi\Big)\|1\|_\frac{p+1}{p}=\| \psi \|_\frac{p+1}{p} +|\Omega|^{-\frac{1}{p+1}}\int_\Omega |\psi|\leq 2\| \psi \|_\frac{p+1}{p}. 
\end{align}
Set $c_k:= \lambda_k|\Omega|^{-1} \int_\Omega |g_k|^{\frac{1}{p}-1}g_k$, and observe that $(c_k)_{k\in\N}$ is bounded in $\R$, so (up to a subsequence) there is $c_0\in\R$ such that
\begin{align}\label{c0}
 c_k\to c_0\qquad \text{ as }k\to \infty.
\end{align}
Then,  by~\eqref{cota} and since $K g_k$ has zero average and $\varphi\in X_{p}$,
\begin{align}
 \int_\Omega  \psi Kg_k + \int_\Omega c_k\psi &- \int_\Omega \lambda_k  |g_k|^{\frac{1}{p}-1}g_k\psi = \int_\Omega \varphi K g_k - \lambda_k \int_\Omega |g_k|^{\frac{1}{p}-1}g_k \varphi\nonumber\\
 &\leq  \|F'(g_k)-\lambda_k G'(g_k)\|_{X_{p}^*} \|\varphi\|_{\frac{p+1}{p}}\leq 2 \|F'(g_k)-\lambda_k G'(g_k)\|_{X_{p}^*} \|\psi\|_{\frac{p+1}{p}}.\label{cota2}
\end{align}
Then, by~\eqref{cota2} and using that $(L^{\frac{p+1}{p}}(\Omega))^*$ is isometrically isomorphic to $L^{p+1}(\Omega)$,
\begin{align}
\|Kg_k-c_k-\lambda_k  |g_k|^{\frac{1}{p}-1}g_k  \|_{p+1}&=\|Kg_k-c_k-\lambda_k  |g_k|^{\frac{1}{p}-1}g_k  \|_{(L^{\frac{p+1}{p}}(\Omega))^*}\nonumber\\
&=\sup_{\psi\in L^{\frac{p+1}{p}}(\Omega)\backslash\{0\} }\frac{|\int_\Omega  (Kg_k + c_k - \lambda_k  |g_k|^{\frac{1}{p}-1}g_k)\psi|}{\|\psi\|_{\frac{p+1}{p}}}\nonumber\\
&\leq 2\|F'(g_k)-\lambda_k G'(g_k)\|_{X_{p}^*}\to 0.\label{mi}
\end{align}
Thus, by~\eqref{ekeland},~\eqref{lambda:c},~\eqref{K:con},~\eqref{c0}, and~\eqref{mi}, we have that $g_k$ must converge a.e. to some function, which is necessarily $g_0$.
\smallbreak

Summarizing the above arguments, we know that $(g_k)$ is a maximizing sequence for $D_{p}$ satisfying
\[
\|g_k-f_k\|_{\frac{p+1}{p}}\to 0, \qquad \|g_k\|_{\frac{p+1}{p}}=1 \quad \text{ for all } k\in\N,\qquad  g_k \to g_0  \text{ a.e. in } \Omega \text{  as $k\to\infty$}.
\]
Write $g_k=g_0+w_k$, with $w_k\rightharpoonup 0$ in $L^\frac{p+1}{p}(\Omega)$  as $k\to\infty$. Then by Brezis-Lieb's Lemma \cite{BL83} (which we can apply due to the a.e. convergence),
\[
\|g_0\|_{\frac{p+1}{p}}^{\frac{p+1}{p}} + \|w_k\|_{\frac{p+1}{p}}^{\frac{p+1}{p}}=1+\text{o}(1),
\]
and then, since $\frac{p+1}{p}<2$ (because $p=2^*-1>1$),
\begin{align*}
 \|g_0\|_{\frac{p+1}{p}}^{2} + \|w_k\|_{\frac{p+1}{p}}^{2}\leq 1+\text{o}(1).
\end{align*}
After this point, we follow closely \cite[p. 1143]{CK91}. Since $g_k$ is a maximizing sequence for $D_p$ and $g_k\rightharpoonup g_0$, $w_k\rightharpoonup 0$, in $L^\frac{p+1}{p}(\Omega)$  as $k\to\infty$, we have
\begin{align*}
D_p+o(1)&=\int_\Omega g_k K g_k=\int_\Omega g_0Kg_0+\int_\Omega w_k Kw_k+\int_\Omega g_0Kw_k+\int_\Omega w_k Kg_0\\
		&=\int_\Omega g_0Kg_0+\int_\Omega w_k Kw_k+o(1)
\end{align*}
and
\begin{align*}
 \int_\Omega g_0 K g_0 + \int_\Omega w_k K w_k = D_p + o(1)\geq D_p(\|g_0\|_{\frac{p+1}{p}}^{2} + \|w_k\|_{\frac{p+1}{p}}^{2}) +o(1).
\end{align*}
By the definition of $D_p$, we have $\int_\Omega g_0 K g_0\leq D_p \|g_0\|_{\frac{p+1}{p}}^2$ and deduce that
\begin{align}\label{38}
D_p \|w_k\|_{\frac{p+1}{p}}^2\leq \int_\Omega w_k K w_k +o(1).
\end{align}
Let $h_k:= K w_k$, by the continuity of $K:L^\frac{p+1}{p}(\Omega)\to H^1(\Omega)$,
\begin{align*}
 h_k\rightharpoonup 0\quad  \text{ in }H^1(\Omega),\qquad 
 h_k \to 0\quad \text{ in }L^2(\Omega),\qquad 
 h_k\rightharpoonup 0\quad  \text{ in }L^{p+1}(\Omega)\quad \text{ as $k\to\infty$}.
\end{align*}

By Cherrier's inequality (see \cite{Ch84} or \cite[equation 3.1]{CK91}) together with H\"older's inequality,
\begin{align*}
 \int_\Omega w_k K w_k \leq \|w_k\|_{\frac{p+1}{p}}\|h_k\|_{p+1} \leq \|w_k\|_{\frac{p+1}{p}}
 \left(
 \frac{2^\frac{2}{N}}{S}+\eps
 \right)^\frac{1}{2}\|\nabla h_k\|_2 + o(1),
\end{align*}
where $S$ is the best Sobolev constant for the embedding $H_0^1(\Omega)\hookrightarrow L^{p+1}(\Omega)$. Then, by~\eqref{38},
\begin{align*}
 D_p \|w_k\|_\frac{p+1}{p}\leq 
 \left(
 \frac{2^\frac{2}{N}}{S}+\eps
 \right)^\frac{1}{2}
 \|\nabla h_k\|_2+o(1).
\end{align*}

By the definition of $D_p$ we know that $\|\nabla h_k\|_2^2=\int_\Omega w_k K w_k  \leq D_p \|w_k\|_{\frac{p+1}{p}}^2$ and therefore
\begin{align*}
 D_p \|w_k\|_\frac{p+1}{p}\leq 
 \left(
 \frac{2^\frac{2}{N}}{S}+\eps
 \right)^\frac{1}{2}
 D_p^\frac{1}{2} \|w_k\|_{\frac{p+1}{p}}+o(1),
\end{align*}
that is,
\begin{align}\label{nlt4}
 \|w_k\|_\frac{p+1}{p}\left(1-\left(
 \frac{2^\frac{2}{N}}{S\, D_p}+\frac{\eps}{D_p}
 \right)^\frac{1}{2}
 \right)\leq 
 o(1).
\end{align}
By \cite[Lemma 2.2]{CK91} we have that $\frac{2^\frac{2}{N}}{S\, D_p}<1$ if $N\geq 4$, and therefore, choosing $\eps>0$ small enough, we deduce that $\|w_k\|_\frac{p+1}{p}=o(1)$ as $k\to\infty$. As a consequence, $g_k\to g_0$ strongly in $L^\frac{p+1}{p}(\Omega)$ as $k\to\infty$.  Finally, since $\|f_k-g_k\|_{\frac{p+1}{p}}\to 0$, we finally conclude that (up to a subsequence) $f_k\to g_0$ in $L^{\frac{p+1}{p}}(\Omega)$.
\end{proof}

\begin{lemma}\label{lemma:D} For every $p\in (0,2^*-1]$ the following holds true. 
\begin{itemize}
\item[i)] There exists $f\in X_p$ with $\| f\|_{\frac{{p}+1}{{p}}}=1$ such that $\int_\Omega f Kf=D_p$.
\item[ii)] Let $f$ be as in i), then 
\[
K_p f=D_p |f|^{\frac{1}{{p}}-1}f \qquad \text{ a.e. in } \Omega.
\]  
\end{itemize}
\end{lemma}
\begin{proof} 
\noindent i) This is a direct consequence of Lemma \ref{lemma:CK91}.

\smallbreak

\noindent ii) Let $f\in X_p$ be such that $\|f\|_{\frac{{p}+1}{{p}}}=1$ and $\int_\Omega f K f= D_p$. Then there is a Lagrange multiplier $\lambda\in \R$ such that
\[
\int_\Omega \varphi K f=\lambda \int_\Omega |f|^{\frac{1}{{p}}-1}f\varphi \qquad \text{ for all }\varphi \in X_p.
\]
Then, if $\psi\in L^{\frac{{p}+1}{{p}}}(\Omega)$, we have $\varphi=\psi-|\Omega|^{-1} \int_\Omega \psi\in X_p$, and hence
\[
\int_\Omega \psi K f + \frac{\lambda}{|\Omega|} \int_\Omega |f|^{\frac{1}{{p}}-1}f \int_\Omega \psi
= \lambda \int_\Omega |f|^{\frac{1}{{p}}-1}f \psi,
\]
by using the fact that $K f$ has zero average. Thus $Kf + c=\lambda |f|^{\frac{1}{{p}}-1}f $ for $c=\lambda |\Omega|^{-1} \int_\Omega |f|^{\frac{1}{{p}}-1}f$, and $|\lambda|^{{p}-1}\lambda f= |Kf+c|^{{p}-1}(K f+c)$. Since $\int_\Omega f=0$, also $\int_\Omega |Kf+c|^{{p}-1}(K f+c)=0$ and thus  $c=c_p(Kf)$, $K_p f=K f+c$.
In conclusion, $K_p f=\lambda |f|^{\frac{1}{{p}}-1}f$. Multiplying this equation by $f$ and integrating yields $\lambda=D_p$.
\end{proof}

\begin{lemma}\label{lemma:convergence_Deps} The map
\begin{align*}
(0,2^*-1]\to \R^+; \qquad p\mapsto D_{p}
\end{align*}
is continuous. Let $p_n,p\in (0,2^*-1]$ be such that $p_n\to p$ as $n\to\infty$. If $f_{p_n}\in X_{p_n}$ achieves $D_{p_n}$, then there exists $f_{p}\in X_{p}$ achieving $D_{p}$ such that, up to a subsequence,
\[
f_{p_n} \to f_{p} \qquad \text{strongly in } L^{\frac{p+1}{p}}(\Omega)\quad \text{ as }n\to\infty.
\]
\end{lemma}
\begin{proof}
Let $p_n\to p$ ($p_n,p\in (0,2^*-1]$) and let $f_p\in X_{p}$ with $\|f_p\|_{\frac{p+1}{p}}=1$ be such that $\int_\Omega f_p K f_p=D_{p}$. By Lemma~\ref{lemma:D}, $u_p:=|f_p|^{\frac{1}{p}-1}f_p$  solves $-\Delta u_p=D_p^{-1}|u_p|^{p-1}u_p$. In particular, by Proposition~\ref{prop:regularity} we have that $u_p\in L^\infty(\Omega)$, and also $f_p=|u_p|^{p-1}u_p\in L^\infty(\Omega)$. Therefore $f_p\in X_{p_n}$ and, by maximality,
\begin{align}\label{Dpb}
 \frac{D_{p}}{\|f_p\|_{\frac{{p_n}+1}{{p_n}}}^2} =\frac{1}{\|f_p\|_{\frac{{p_n}+1}{{p_n}}}^2} \int_\Omega f_p K f_p=\int_\Omega \frac{f_p}{\|f_p\|_{\frac{{p_n}+1}{{p_n}}}} K \left(\frac{f_p}{\|f_p\|_{\frac{{p_n}+1}{{p_n}}}}  \right) \leq D_{p_n}.
\end{align}
Since $\|f_p\|_{\frac{{p_n}+1}{{p_n}}}^2\to \|f_p\|_{\frac{p+1}{p}}^2=1$ (by Lebesgue's dominated convergence theorem), then $D_{p}\leq \liminf\limits_{n\to\infty} D_{p_n}$.

\medskip

Let $f_{p_n}$ achieve $D_{p_n}$, that is: $f_{p_n}\in X_{p_n}$ with $\|f_{p_n}\|_{\frac{{p_n}+1}{{p_n}}}=1$ and $\int_\Omega f_{p_n} K f_{p_n} = D_{p_n}$.  Passing to a subsequence, we consider two cases: 

\medskip

\textbf{Critical case $p=2^*-1$:} In this case, $\frac{p+1}{p}<\frac{{p_n}+1}{{p_n}}$ and, since $\Omega$ is bounded, $L^{\frac{{p_n}+1}{{p_n}}}(\Omega) \subseteq L^\frac{p+1}{p}(\Omega)$. Then $\frac{f_{p_n}}{\|f_{p_n}\|_{\frac{p+1}{p}}}\in X_{p}$ and 
\begin{equation}\label{eq:auxiliary_convergence}
\frac{D_{p_n}}{\|f_{p_n}\|_{\frac{p+1}{p}}^2} =\frac{1}{\|f_{p_n}\|_{\frac{p+1}{p}}^2} \int_\Omega f_{p_n} K f_{p_n}=\int_\Omega \frac{f_{p_n}}{\|f_{p_n}\|_{\frac{p+1}{p}}} K \left(\frac{f_{p_n}}{\|f_{p_n}\|_{\frac{p+1}{p}}}  \right) \leq D_{p}.
\end{equation}
From H\"older's inequality we have
\begin{equation}\label{eq:Holder}
\|f_{p_n}\|_{\frac{p+1}{p}}^2\leq \|f_{p_n}\|_{\frac{{p_n}+1}{{p_n}}}^2|\Omega|^\frac{2(p-{p_n})}{(p+1)({p_n}+1)} =|\Omega|^\frac{2(p-{p_n})}{(p+1)({p_n}+1)}=1+o(1)\quad \text{ as }n\to\infty,
\end{equation}
thus $\limsup\limits_{n\to\infty} D_{p_n}\leq D_{p}
\limsup\limits_{n\to\infty}\|f_{p_n}\|_{\frac{p+1}{p}}^2
\leq D_{p}$  and so $D_{p_n} \to D_{p}$. Going back to~\eqref{eq:auxiliary_convergence}, we see that actually $\|f_{p_n}\|_{\frac{p+1}{p}}^2\to 1$, and $(\frac{f_{p_n}}{\|f_{p_n}\|_{\frac{p+1}{p}}})$ is a maximizing sequence for $D_{p}$. Therefore, by Lemma~\ref{lemma:CK91}, there exists $f_{p}\in L^\frac{p+1}{p}(\Omega)$ such that
\[
\frac{f_{p_n}}{\|f_{p_n}\|_{\frac{p+1}{p}}} \to f_{p} \qquad \text{ strongly  in } L^\frac{p+1}{p}(\Omega)\quad \text{ as }n\to\infty.
\]
Using again that $\|f_{p_n}\|_{\frac{p+1}{p}} \to 1$, the conclusion follows.

\medskip

\textbf{Subcritical case $p<2^*-1$:}  To simplify notation, we use $c>1$ to denote possibly different constants independent of $n$ and we often pass to subsequences without relabeling. Since $f_{p_n}$ is a maximizer for $D_{p_n}$, we have by Lemma \ref{lemma:D} that
\begin{align}\label{undef}
 u_{p_n}:= |f_{p_n}|^{\frac{1}{p_n}-1}f_{p_n}=(D_{p_n})^{-1}K_{p_n} f_{p_n}
\end{align}
From this, we have 
\[
\|\nabla u_{p_n}\|_2=D_{p_n}^{-1} \left(\int_\Omega |\nabla K_{p_n} f_{p_n}|^2 \right)^{1/2}=D_{p_n}^{-1} \left(\int_\Omega |f_{p_n}|^\frac{p_n+1}{p_n} \right)^{1/2}=D_{p_n}^{-1}  <1/\delta.
\]
 where we used Lemma \ref{lemma:bounds} for the bound. Moreover, by \eqref{undef},
\[
\left|\int_\Omega u_{p_n}\right|\leq \int_\Omega |f_{p_n}|^\frac{1}{p_n}\leq |\Omega|^\frac{p_n}{p_n+1}\leq c.
\]
Therefore, by Poincar\'e-Wirtinger's  inequality,
\[
\|u_{p_n}\|_2 \leq \left\| u_{p_n}-\frac{1}{|\Omega|}\int_\Omega u_{p_n}  \right\|_2 + \left\| \frac{1}{|\Omega|}\int_\Omega u_{p_n} \right\|_2\leq c\|\nabla u_{p_n}\|_2+c\leq c.
\]
As a consequence, by Sobolev embeddings,
\begin{align}\label{embed}
(u_{p_n}) \text{ is bounded in }L^{q}(\Omega)\quad \text{ for any }q\in(1,2^*). 
\end{align}
But then, again by~\eqref{undef}, we have that $|f_{p_n}|= |u_n|^{p_n}$ and, since $p_n\to p\in(0,2^*-1)$,
 \begin{align*}
|f_{p_n}|^{\frac{p_n+1}{p_n}}= |u_{p_n}|^{p_n+1}\in L^1(\Omega)
 \end{align*}
and therefore, by~\eqref{embed} and Lemma~\ref{lemma:D}, 
\begin{align}\label{fn1}
 \lim_{n\to\infty}\|f_{p_n}\|^{\frac{p+1}{p}}_{\frac{p+1}{p}}
 =\lim_{n\to\infty}\|u_{p_n}\|_{p+1}^{p+1}=\lim_{n\to\infty}\|u_{p_n}\|_{p_n+1}^{p_n+1}
 =\lim_{n\to\infty}\|f_{p_n}\|_{\frac{p_n+1}{p_n}}=1,
\end{align}
where we used the fact that there is $\delta>0$ small such that
\begin{align*}
s(p+1)+(1-s)(p_n+1)\in(p+1-\delta,p+1+\delta),\quad \text{ for } n\text{ large, }s\in(0,1),
\end{align*}
 and
\begin{align*}
 \left|
 \|u_{p_n}\|_{p+1}^{p+1}-\|u_{p_n}\|_{p_n+1}^{p_n+1}
 \right|
&\leq |p_n-p|\int_\Omega\int_0^1 |u_{p_n}|^{s(p+1)+(1-s)(p_n+1)}|\ln(|u_{p_n}|)|\ ds\, dx\\
&\leq c |p_n-p|\int_\Omega \left(|u_{p_n}|^{p+1-\delta}+|u_{p_n}|^{p+1+\delta}\right)\, dx=o(1)
\end{align*}
as $n\to\infty$.   Then
\begin{align*}
  \lim_{n\to \infty}D_{p_n}=\lim_{n\to\infty}\int_\Omega f_{p_n} K f_{p_n}=\lim_{n\to\infty}\frac{\int_\Omega f_{p_n} K f_{p_n}}{\|f_{p_n}\|^2_{\frac{p+1}{p}}}\leq D_{p}.
\end{align*}
Using~\eqref{Dpb}, it follows that $\lim_{n\to \infty}D_{p_n}=D_p$. Since 
$f_{p_n}/\|f_{p_n}\|_{\frac{p+1}{p}}$ is a maximizing sequence for $D_p,$ there is, by Lemma~\ref{lemma:CK91}, $\widehat f\in L^{\frac{p+1}{p}}(\Omega)$ such that $f_{p_n}\to \widehat f$ in $L^{\frac{p+1}{p}}(\Omega)$ and (by the continuity of $f\mapsto fKf$, see \eqref{eq:wc}) $\widehat f$ is a maximizer for $D_p$.
\end{proof}

\subsection{Results regarding $\Lambda_p$ ($p>0$)}  \label{sec:Lambda_p}
 
 In this subsection we translate the results obtained in the dual framework in terms of the direct functional $I_p$, proving Theorem \ref{th:NeumannLENS_0}. The main step is to prove that $\Lambda_p=D_p^{-1}$. The case $p=2^*-1$ is not straightforward, and we first have to deal with the subcritical case $p\in (0,2^*-1)$. 
 
 Recall the map $c_p: L^{p+1}(\Omega)\to \R$ (defined in \eqref{Ks:def}), where $c_p(w)$ is the unique real number such that $\int_\Omega |w+c_p(w)|^{p-1}(w+c_p(w))=0$ for $w\in L^{p+1}(\Omega)$. 
 
The next Lemma is a direct consequence of \cite[Proof of Lemma 2.1]{PW15}.
\begin{lemma}\label{eq:auxlemmaPW}
For $p\geq0$, the map $c_p: L^{p+1}(\Omega)\to \R;$ $u\mapsto c_p(u)$ is continuous and, for every $w\in L^{p+1}(\Omega)$,
 \[
 \|w+c_p(w)\|_{p+1}=\min_{c\in \R} \|w+c\|_{p+1}.
\]
\end{lemma}

\begin{lemma}\label{lemma:Lambdap_achieved}
Let $p\in (0,2^*-1)$. Then $\Lambda_p$ is achieved. Moreover, every minimizer $v_p$ satisfies
\begin{equation}\label{eq:v_p}
-\Delta v_p=\Lambda_p |v_p|^{p-1}v_p \text{ in } \Omega,\quad \partial_\nu v_p=0 \text{ on } \partial \Omega.
\end{equation}
\begin{proof}
Since the embedding $H^1(\Omega)\hookrightarrow L^p(\Omega)$ is compact, using the direct method of Calculus of Variations it is straightforward to check that $\Lambda_p$ is achieved. Let $v$ be a minimizer. If $p\geq 1$, then by the theory of Lagrange multipliers we conclude that $v$ solves \eqref{eq:v_p}. We now focus on the case $p\in (0,1)$. Inspired by the proof of \cite[Lemma 2.2]{PW15}, assume in view of a contradiction the existence of $\varphi\in H^1(\Omega)$ such that
\begin{equation}\label{eq:contra_aux}
\int_\Omega \nabla v \cdot \nabla \varphi - \Lambda_p \int_\Omega |v|^{p-1}v \varphi<0.
\end{equation}
Consider the test function
\[
w_t:= \frac{v+t\varphi +c_p(v+t\varphi)}{\|v+t\varphi +c_p(v+t\varphi)\|_{p+1}}.
\]
By Lemma \ref{eq:auxlemmaPW} and since $c_p(v)=0$, we have $c_p(v+t\varphi)\to 0$ as $t\to 0$ and $1=\|v\|_{p+1}^2\leq \|v+c_p(v+t\varphi)\|_{p+1}$.  For $t>0$ sufficiently small, we apply the intermediate value theorem to the map $h(s):=\|v+s\varphi +c_p(v+t\varphi)\|^{-2}_{p+1}$ between $s=0$ and $s=t$, obtaining
\begin{align*}
\frac{1}{\|v+t\varphi +c_p(v+t\varphi)\|^2_{p+1}}-\frac{1}{\|v +c_p(v+t\varphi)\|^2_{p+1}}=h(t)-h(0)=h'(s_t)t
\end{align*}
for some $s_t\in (0,t)$. Therefore,
\begin{align*}
\frac{1}{\|v+t\varphi +c_p(v+t\varphi)\|^2_{p+1}}\leq 1-2t\frac{ \int_\Omega |v+s_t \varphi+ c_p(v+t\varphi)|^{p-1}(v+c_p(v+t\varphi)+s_t\varphi) \varphi}{\|v+c_p(v+t\varphi )+ s_t \varphi\|_{p+1}^{p+3}}.
\end{align*}
From $\|v\|_{p+1}=1$ and  the fact that $c_p(v+t\varphi)\to 0$ as $t\to 0$, then
\begin{equation*}
\frac{ \int_\Omega |v+s_t \varphi+ c_p(v+t\varphi)|^{p-1}(v+c_p(v+t\varphi)+s_t\varphi) \varphi}{\|v+c_p(v+t\varphi )+ s_t \varphi\|_{p+1}^{p+3}}\to \frac{\int_\Omega |v|^{p-1}v\varphi}{\|v\|_{p+1}^{p+3}}=\int_\Omega |v|^{p-1}v \varphi
\end{equation*}
which implies that the first quantity is bounded, and
\begin{align*}
\int_\Omega |\nabla v|^2&\leq \int_\Omega \left|\nabla w_t\right|^2=\frac{1}{\|v+t\varphi +c_p(v+t\varphi)\|_{p+1}^2}\left(\Lambda_p + 2t \int_\Omega \nabla v\cdot \nabla \varphi +o(t) \right)\\
				      &\leq \Lambda_p + 2t \left(\int_\Omega \nabla v\cdot \nabla \varphi  - \Lambda_p \frac{ \int_\Omega |v+c_p(v+t\varphi)|^{p-1}(v+c_p(v+t\varphi)) \varphi}{\|v+c_p(v+t\varphi)\|_{p+1}^{p+3}} \right) + o(t)<\Lambda_p
\end{align*}
for sufficiently small $t>0$, which is a contradiction. Observe that the last inequality is a consequence of \eqref{eq:contra_aux}, which yield
\[
\int_\Omega \nabla v\cdot \nabla \varphi  - \Lambda_p \frac{ \int_\Omega |v+c_p(v+t\varphi)|^{p-1}(v+c_p(v+t\varphi)) \varphi}{\|v+c_p(v+t\varphi)\|_{p+1}^{p+3}} \to \int_\Omega \nabla v \cdot \nabla \varphi - \Lambda_p \int_\Omega |v|^{p-1}v \varphi<0
\]
as $t\to 0$. In conclusion, $v$ satisfies \eqref{eq:v_p}.

\end{proof}
\end{lemma}

\begin{lemma}\label{lemma:DpLambdap}
We have
\begin{align}\label{r}
\Lambda_p= \frac{1}{D_p} \qquad \text{ for every } p\in (0,2^*-1),
\end{align}
and $f$ is a maximizer for $D_p$ if and only if $v=|f|^{\frac{1}{p}-1}f$ is a minimizer for $\Lambda_p$. Moreover,
\[
\Lambda_{2^*-1}\leq \frac{1}{D_{2^*-1}}.
\]
\end{lemma}
\begin{proof}
Let $p\in (0,2^*-1]$. By Lemma \ref{lemma:D}, there is
$f\in X_p$ such that
\[
\|f\|_\frac{p+1}{p}=1,\qquad \int_\Omega fKf=D_p.
\]
Let $u:=K_p f=D_p |f|^{\frac{1}{p}-1}f$ and $v:=\frac{u}{\|u\|_{p+1}}\in H^1(\Omega)$, which satisfies $\|v\|_{p+1}=1$ and $\int_\Omega |v|^{p-1}v=0$. Then
\begin{align*}
\Lambda_p\leq \int_\Omega |\nabla v|^2=\frac{\| \nabla u\|_2^2}{\|u\|^2_{p+1}}= \frac{\int_\Omega |\nabla K_pf|^2}{\left(\int_\Omega |K_pf|^{p+1}\right)^\frac{2}{p+1}}=\frac{\int_\Omega f K_p f}{D_p^2}=\frac{1}{D_p}.
\end{align*}

Assume moreover that $p<2^*-1$. Let $u$ be an extremizer for  $\Lambda_p$, which solves
\[
-\Delta u=\Lambda_p |u|^{p-1}u \text{ in } \Omega, \qquad \partial_\nu u=0 \text{ in } \partial \Omega,\qquad \|u\|_{p+1}=1.
\]
(recall Lemma \ref{lemma:Lambdap_achieved}). Let $f=|u|^{p-1}u$, which satisfies $\int_\Omega |f|^\frac{p+1}{p}=\int_\Omega |u|^{p+1}=1$. Then
\begin{align*}
D_p \geq \int_\Omega fKf =\int_\Omega |u|^{p-1} u K(|u|^{p-1}u)=\frac{1}{\Lambda_p}\int_\Omega |u|^{p+1}=\frac{1}{\Lambda_p},
\end{align*}
which ends the proof.
\end{proof}

Using the previous lemmas and the continuity of $D_p$ up to $p=2^*-1$ (Lemma \ref{lemma:convergence_Deps}) we show next the validity of \eqref{r} for $p=2^*-1$.

\begin{lemma}\label{lemma:D_p=1/Lambda}
It holds that
\[
\Lambda_{2^*-1}=\frac{1}{D_{2^*-1}}.
\]
Moreover, $f$ is a maximizer for $D_{2^*-1}$ if and only if $v=|f|^{\frac{1}{p}-1}f$ is a minimizer for $\Lambda_{2^*-1}$.
\end{lemma}

\begin{proof}
By Lemmas \ref{lemma:convergence_Deps} and \ref{lemma:DpLambdap},
\[
\Lambda_{2^*-1}\leq \frac{1}{D_{2^*-1}}= \lim_{p\nearrow 2^*-1} \frac{1}{D_p}=\lim_{p\nearrow 2^*-1} \Lambda_{p}.
\]
We now argue by contradiction. Assume that
\begin{align*}
 \lim_{p\to 2^*-1}\Lambda_p> \Lambda_{2^*-1}=\inf_{v\in H^1(\Omega),\int_\Omega |v|^{2^*-2}v=0} \frac{\|\nabla v\|_2^2}{\|v\|_{2^*}^{2}}.
\end{align*}
Therefore, there is $w\in H^1(\Omega)\backslash\{0\}$ such that $\int_\Omega |w|^{2^*-2}w=0$ and
\begin{align*}
 a:=\lim_{p\nearrow 2^*-1}\Lambda_p> \frac{\|\nabla w\|_2^2}{\|w\|_{2^*}^{2}}=:b>0.
\end{align*}
But then, there is a sequence $p_n \nearrow 2^*-1$ and $\delta:=\frac{a-b}{2}>0$ satisfying
\begin{align*}
 \Lambda_{p_n}
 > \frac{\|\nabla w\|_2^2}{\|w\|_{2^*}^{2}}+\delta
 =\frac{\|\nabla w\|_2^2}{\|w+c_n\|_{p_n+1}^{2}}\frac{\|w+c_n\|_{p_n+1}^{2}}{\|w\|_{2^*}^{2}}+\delta
 \geq \Lambda_{p_n}\frac{\|w+c_n\|_{p_n+1}^{2}}{\|w\|_{2^*}^{2}}+\delta
 \qquad \text{ for all }n\in\N,
\end{align*}
where $c_n=c_{p_n}(w)\in\R$ is such that $\int_\Omega|w+c_n|^{p_n-1}(w+c_n)=0$.  We claim that 
\[
\|w+c_n\|_{p_n+1}\to \|w\|_{2^*}\text{ as } n\to\infty. 
\]
Once we know this is true, we obtain
\begin{align*}
 \Lambda_{p_n}
 > \Lambda_{p_n}(1+o(1))+\delta=\Lambda_{p_n}+o(1)+\delta>\Lambda_{p_n},
\end{align*}
a contradiction. Therefore $\Lambda_{2^*-1}=\frac{1}{D_{2^*-1}}$, and the second statement of the lemma follows from Lemma \ref{lemma:D}-(ii).

Therefore, we are left with the proof of the claim. Observe that, by Lemma \ref{eq:auxlemmaPW},
 \begin{align*}
  \int_\Omega|w+c_n|^{p_n+1}=\|w+c_n\|_{p_n+1}^{p_n+1}=\min_{c\in\R}\|w+c\|_{p_n+1}^{p_n+1}\leq \|w\|_{p_n+1}^{p_n+1}=\int_\Omega |w|^{p_n+1}\leq \int_\Omega |w+1|^{2^*}.
 \end{align*}
This shows that $c_n$ is bounded. Then, by dominated convergence, $|w+c_n|_{p_n+1}\to |w+c^*|_{2^*}$ as $n\to\infty$, where $c^*=\lim_{n\to\infty}c_n.$ Then, again by dominated convergence,
\begin{align*}
0=\lim_{n\to\infty}\int_\Omega |w+c_n|^{p_n-1}(w+c_n)= \int_\Omega |w+c^*|^{2^*-2}(w+c^*),
\end{align*}
and therefore $c^*=c_{2^*-1}(w)=0$.
\end{proof}

We are now ready to conclude the proofs of the results stated in the introduction regarding $\Lambda_p$.
\begin{proof}[Proof of Theorem \ref{th:NeumannLENS_0}] We split the proof into three parts. 
\smallbreak

\noindent 1) Combining Lemma \ref{lemma:convergence_Deps} (in the dual framework) with Lemmas \ref{lemma:Lambdap_achieved}--\ref{lemma:D_p=1/Lambda},  we conclude $\Lambda_p$ is achieved for $p\in (0,2^*-1]$, extremizers satisfy $-\Delta v_p=\Lambda_p |v_p|^{p-1}v_p$ in $\Omega$ with Neumann b.c., and the map $p\mapsto \Lambda_p$ is continuous.

\smallbreak

\noindent 2) Let $p_n,p \subset (0,2^*-1]$ be such that $p_n\to p$ as $n\to\infty$, and let $v_{p_n}$ be a minimizer for $\Lambda_{p_n}$. By Lemmas \ref{lemma:DpLambdap}--\ref{lemma:D_p=1/Lambda},   $f_{p_n}:=|v_{p_n}|^{p_n-1}v_{p_n}$ is a maximizer for $D_{p_n}$, and by Lemma \ref{lemma:convergence_Deps} there exists $f_{p}\in X_{p}$ achieving $D_{p}$ such that, up to a subsequence,
\[
f_{p_n} \to f_{p} \qquad \text{strongly in } L^{\frac{p+1}{p}}(\Omega)\quad \text{ as }n\to\infty.
\]
Then $v_p:=|f_p|^{\frac{1}{p}-1}f_p$ achieves $\Lambda_p$.
From the converse of the dominated convergence theorem (see for instance \cite[Lemma A.1]{W96}), up to a subsequence we have $|f_{p_n}|\leq h\in L^\frac{p+1}{p}$. Therefore $|v_{p_n}|\leq h^\frac{1}{p}\in L^{p+1}(\Omega)$ and, from dominated convergence,
\[
v_{p_n}\to v_p \text{ strongly in } L^{p+1}(\Omega).
\]

If $0<p<2^*-1$, then, by Lemma~\ref{l:reg}, the sequence $(v_{p_n})_{n\in\N}$ is uniformly bounded in $W^{2,t}(\Omega)$, and, by Sobolev embeddings, $v_{p_n}\to v_p$ in $C^{0,\alpha}(\overline \Omega)$ (up to a subsequence) for any exponent $\alpha\in (0,1)$. Then, by Schauder estimates \cite[Theorem 6.30]{GT98}, $(v_{p_n})_{n\in\N}$ is uniformly bounded in $C^{2,\alpha}(\overline \Omega)$ and since $C^{2,\alpha}(\overline \Omega)\hookrightarrow C^{2,\alpha'}(\overline \Omega)$ is compact whenever $0<\alpha'<\alpha$, we obtain that
\[
v_{p_n} \to v_{p} \quad \text{ in  } C^{2,\alpha'}(\overline \Omega)\quad  \text{ (up to a subsequence)}\quad \text{ as }n\to\infty.
\]
As for the case $p=2^*-1$, observe that
\[
-\Delta v_{p_n} =h_{p_n}(x) v_{p_n},
\]
with $h_{p_n}(x)=\Lambda_{p_n}|v_{p_n}|^{p_n-1} \leq C|h|^\frac{p_n-1}{2^*-1}\leq (1+|h|^\frac{2^*-2}{2^*-1})\in L^\frac{N}{2}(\Omega)$. Then, by Lemma~\ref{lemma:regularitysol}, the sequence $(u_{p_n})$ is bounded in $L^t(\Omega)$ for every $t\geq 1$ and  also  in $W^{2,t}(\Omega)$ for $t\geq 1$ (by Lemma~\ref{l:reg}). Reasoning as before, we have convergence in $C^{2,\alpha}(\overline \Omega)$ for any $\alpha\in (0,1)$.

\smallbreak

\noindent 3) Let $u_p$ be a minimizer of $\Lambda_p$, and let $\psi_1$ be an $L^2$--normalized eigenfunction (i.e. a minimizer of $\Lambda_1$) such that $u_p\to \psi_1$ strongly in $H^1(\Omega)$. We have
\begin{align*}
 \int_\Omega \mu_1 u_p \psi_{1}=\int_\Omega \nabla u_p \nabla \psi_{1}
 =\int_\Omega \Lambda_p |u_p|^{p-1}u_p \psi_{1}
\end{align*}
and therefore
\begin{align*}
 0&=\int_\Omega u_p \psi_{1} \left(\frac{\Lambda_p}{\mu_1}|u_p|^{p-1}-1\right) = \int_\Omega u_p \psi_1 \int_0^1 \frac{d}{ds}  \left\{\left(\left(\frac{\Lambda_p}{\mu_1}\right)^\frac{1}{p-1}|\mu_1| \right)^{s(p-1)}\right\} \ dsdx\\
& =(p-1)\int_\Omega u_p \psi_{1} \int_0^1 \ln
 \left(
 \left(
 \frac{\Lambda_p}{\mu_1}\right)^\frac{1}{p-1}|u_p|
  \right)
  \left|
  \left(
 \frac{\Lambda_p}{\mu_1}\right)^\frac{1}{p-1}|u_p|
 \right|^{s(p-1)}\ ds\, dx\\
 &=(p-1)\int_\Omega\int_0^1
 u_p \psi_{1}
 \left(\frac{1}{p-1}
 \ln
 \left(
  \frac{\Lambda_p}{\mu_1}
  \right)
  +
  \ln
  |u_p|
   \right)
  \left(
 \frac{\Lambda_p}{\mu_1}\right)^s
 \left|
 u_p
 \right|^{s(p-1)}\ ds\, dx.
\end{align*} 
Then, by the strong convergence of $u_p$ to $\psi_1$,
\begin{align*}
 \int_\Omega\int_0^1
 u_p \psi_{1}
 \left(
 \frac{\Lambda_p}{\mu_1}\right)^s
 \left|
 u_p
 \right|^{s(p-1)}\ ds\, dx>0
\end{align*}
for all $p$ sufficiently close to $1$ and 
\begin{align*}
\frac{ \ln
 \left(
  \frac{\Lambda_p}{\mu_1}
  \right)}{p-1}
  =-\frac{
  \int_\Omega\int_0^1
 u_p \psi_{1}\left(
  \ln
  |u_p|
   \right)
  \left(
 \frac{\Lambda_p}{\mu_1}\right)^s
 \left|
 u_p
 \right|^{s(p-1)}\ ds\, dx}{\int_\Omega\int_0^1
 u_p \psi_{1}
 \left(
 \frac{\Lambda_p}{\mu_1}\right)^s
 \left|
 u_p
 \right|^{s(p-1)}\ ds\, dx},
\end{align*}
that is, again by strong convergence,
\begin{align*}
\ln\left(
 \left(
  \frac{\Lambda_p}{\mu_1}
  \right)^\frac{1}{p-1}\right)
  =
\frac{ \ln
 \left(
  \frac{\Lambda_p}{\mu_1}
  \right)}{p-1}
  =-\frac{\int_\Omega
 \psi_{1}^2\left(
  \ln
  |\psi_{1}|
   \right)
  \, dx}{\int_\Omega
 \psi_{1}^2 \, dx}
+o(1)
=-\frac{1}{2}\int_\Omega
 \psi_{1}^2\ln(\psi_{1}^2)+o(1).
\end{align*}
and the proof is finished.

\end{proof}

\begin{remark}\label{bbgv:rmk}
In \cite{BBGV} the authors consider the limit as $p\searrow 1$ of problem
\begin{align*}
 -\Delta u_p = \lambda_2 |u_p|^{p-1}u_p\quad \text{ in }\Omega,\qquad 
 u_p=0\quad \text{ on }\partial \Omega,
\end{align*}
where $\lambda_2$ is the second Dirichlet eigenvalue of the Laplacian in $\Omega$, which is a bounded smooth domain in $\R^N$.  One of their main results \cite[Theorem 4]{BBGV} guarantees that $u_p\to u_*$ in $H^1_0(\Omega)$ as $p\searrow 2$, where $u_*$ is an eigenfunction associated to $\lambda_2$ and such that 
\begin{align}\label{B:cond}
 \int_\Omega u_*^2 \ln(u_*^2) = 0.
\end{align}
Unfortunately, the proof of \cite{BBGV} does not extend directly to the Neumann case.  One of the main obstacles is \cite[Lemma 4.4]{BBGV}, used to show that the limit $u_p$ is nontrivial, which relies on the Poincar\'e inequality ($\|\nabla u_p\|_2\geq \lambda_2 \|u_p\|_2^2$).  The analogue of this result for the Neumann problem is the Poincar\'e-Wirtinger inequality ($\|\nabla u_p\|_2\geq \lambda_2 \|u_p-\overline{u}\|_2^2$, with $\bar u=\int_\Omega u$); however, the presence of a possibly nontrivial $\overline{u}$ does not seems to allow to conclude as in as in \cite[Lemma 4.4]{BBGV}.  Furthermore, the arguments in \cite{BBGV} only apply to superlinear problems (see \cite[Lemma 4.1]{BBGV}).  On the other hand, \cite[Lemma 4.3]{BBGV} can be extended to the Neumann case with only minor changes.  As a consequence, our Theorem \ref{th:NeumannLENS_0} implies that, if $u_1 = e^{-\frac{1}{2}\int_\Omega \psi_{1}^2\ln(\psi_{1}^2)}\psi_{1}$ as in \eqref{u1}, then $u_1$ satisfies~\eqref{B:cond} and it can also be indirectly characterized as the minimizer of a particular functional, as in \cite[Theorem 4]{BBGV}.
\end{remark}

\begin{remark}
 In the Dirichlet problem, bifurcation theory has been used to show that there is a continuous branch of solutions $\{v_p:p\in(0,2^*-1)\}$, see for example \cite[Section 7.4.1]{MW05}. However, this approach does not guarantee that the branch consists of least-energy solutions.
\end{remark}

\begin{remark}\label{GW:rmk}
In \cite{GiraoWeth}, the authors study existence and qualitative properties of extremal functions for Poincar\'e-Sobolev-type inequalities, see \eqref{GW:pb}. This problem is related to ours, but in our case we require the nonlinear condition $\int_\Omega |u|^{p-1}u=0$, see \eqref{Lambdap:def}.  As a consequence, many of the techniques in \cite{GiraoWeth} for problem \eqref{GW:pb} cannot be easily extended to \eqref{Lambdap:def}.  For instance, in \cite{GiraoWeth} it is shown the existence of minimizers for $N=3$ and $p=2^*-1$. This relies on the crucial estimate \cite[Proposition 2.1]{GiraoWeth}, which uses the linearity of the condition $\int_\Omega u=0$ in \eqref{GW:pb}. In fact, the same argument fails when considering the condition $\int_\Omega |u|^{p-1}u=0$, since some remainder terms do not decay sufficiently fast to zero in this setting. We recall that existence of l.e.n.s. for \eqref{Lambdap:def} is an open problem for $N=3$ and $p=2^*-1$.
\end{remark}

\subsection{Results regarding $L_p$ ($p>0$)} \label{sec:L_p}
We start by proving Proposition \ref{prop:rel_Lp_Dp}, which connects $\Lambda_p$ with $L_p$.
\begin{proof}[Proof of Proposition \ref{prop:rel_Lp_Dp}] Let $p\in(0,2^*-1]\backslash\{1\}$ and let $v\in H^1(\Omega)$ be a minimizer for $\Lambda_p$ which, by Theorem \ref{th:NeumannLENS_0}, exists and satisfies $-\Delta v=\Lambda_p |v|^{p-1}v$ in $\Omega$, $\partial_\nu v=0$ on $\partial \Omega$. Then $u:=\Lambda_p^\frac{1}{p-1}v$ satisfies \eqref{eq:LENSNeumann}, hence
\begin{align*}
L_p\leq I_p(u)=\frac{p-1}{2(p+1)}\|\nabla u\|_2^2= \frac{p-1}{2(p+1)}\Lambda_p^\frac{p+1}{p-1}.
\end{align*}
Conversely, since the set of solutions is nonempty by the previous paragraph, for any solution $u$ of \eqref{eq:LENSNeumann} we have that
\begin{align*}
\Lambda_p\leq  \int_\Omega \left|\nabla \frac{u}{\|u\|_{p+1}}\right|^2=\frac{\|\nabla u\|_2^2}{\|u\|_{p+1}^2}=\|\nabla u\|_2^\frac{2(p-1)}{p+1}=\left(\frac{2(p+1)}{p-1}I_p(u)\right)^\frac{p-1}{p+1}.
\end{align*}
By taking the infimum in $u$, the result follows.
\end{proof} 

\begin{proof}[Proof of Theorem \ref{th:NeumannLENS}--2,3,4] The proof of the statements in this theorem for $p>0$ are now a direct consequence of Theorem \ref{th:NeumannLENS_0} and Proposition \ref{prop:rel_Lp_Dp}.
\end{proof}

\begin{proof}[Proof of Proposition \ref{thm:varcha}] 
\noindent 1) For the case $p=0$, see \cite{PW15}.

\medskip

\noindent 2) Let $m_p:=\inf\{I_p(u):\ u\in H^1(\Omega)\setminus \{0\},\ \int_\Omega |u|^{p-1}u=0\}$. Since $L_p$ is achieved (by Proposition \ref{prop:rel_Lp_Dp}), there exists $u_p$ solution to \eqref{eq:LENSNeumann} such that $I_p(u_p)=L_p$. By Proposition \ref{prop:regularity}, $u_p$ is of class $C^2$; hence, by integrating the equation we see that $\int_\Omega |u|^{p-1}u=0$, which yields $L_p=I_p(u_p)\geq m_p.$

Regarding the inequality $m_p \geq L_p$, this is a direct consequence of \cite[Lemma 2.2]{PW15}, where it is shown that $m_p$ is a critical level. To keep this paper as self contained as possible, here we provide an alternative argument. Let $u\in H^1(\Omega)\setminus \{0\}$ be such that $\int_\Omega |u|^{p-1}u=0$. Then $\|u\|_{p+1}^2 \Lambda_p \leq \|\nabla u\|_2^2$ by definition of $\Lambda_p$, and
\[
I_p(u)=\frac{1}{2}\|\nabla u\|_2^2-\frac{1}{p+1}\|u\|_{p+1}^{p+1}\geq \varphi(\|u\|_{p+1}),
\]
where $\varphi(t):=\frac{\Lambda_p}{2}t^2-\frac{1}{p+1}t^{p+1}$ has a global minimum at $t=\Lambda_p^\frac{1}{p+1}$. Hence $I_p(u)\geq (\frac{1}{2}-\frac{1}{p+1})\Lambda_p^\frac{p+1}{p-1}$, and the proof now follows from Proposition \ref{prop:rel_Lp_Dp}.

\smallbreak

\noindent 3) Let $n_p=\inf_{w\in \mathcal{M}_p} I_p(w)$. Clearly $L_p\geq n_p$. As for the reverse inequality, given $u\in \mathcal{M}_p$ we have
\[
\Lambda_p\leq \frac{\|\nabla u\|_2^2}{\|u\|_{p+1}^2}=\|\nabla u\|_2^{\frac{2(p-1)}{p+1}} \iff \| \nabla u\|_2^2\geq \Lambda_p^\frac{p+1}{p-1}.
\]
Therefore $I_p(u)\geq  (\frac{1}{2}-\frac{1}{p+1})\Lambda_p^\frac{p+1}{p-1}$. The proof now follows from Proposition \ref{prop:rel_Lp_Dp} and standard arguments.
\end{proof}

\begin{remark}
In the dual framework, Proposition \ref{thm:varcha}-2,3 reads as follows:
\begin{enumerate}
\item for $p\in (0,1)$, 
\[
L_p=\inf\{\phi_p(f):\ f\in X_p\}.
\]
\item for $p\in (1,2^*-1]$,
\[
L_p= \inf_{w\in \mathcal{N}_p} \phi_p(w)= \inf_{f\in X_p\setminus\{0\}} \sup_{t>0} \phi_p(tw),
\]
where $\mathcal{N}_p$ is the Nehari-type set:
\[
\mathcal{N}_p=\left\{f\in X_p\setminus\{0\}: \phi_p'(f)f=0  \right\}.
\]
\end{enumerate}
\end{remark}

\section{Limit at \texorpdfstring{$p=0$}{1}}\label{p0:sec}

In this section we prove Part 1 of Theorem \ref{th:NeumannLENS}, namely, that the solution $u_p$ of~\eqref{eq:LENSNeumann} converges in $C^{1,\alpha}(\Omega)$ as $p\searrow 0$ to a function $u_0\in H^1(\Omega)$ which is a (weak) solution of 
\begin{align}\label{P0}
 -\Delta u_0 = \sign(u_0) \ \text{ in }\Omega, \qquad 
\partial_\nu u_0=0  \text{ on }\partial \Omega,
\end{align}
where
\begin{align*}
 \sign(u)
 =\begin{cases}
      1, &\text{ if }u>0,\\
      -1, &\text{ if }u<0,\\
      0, &\text{ if }u=0.
  \end{cases}
\end{align*}
Since this nonlinearity is no longer invertible, we cannot use a dual characterization. In this case we rely entirely on a direct approach. 

The solution of~\eqref{P0} can be found variationally (see \cite{PW15}) as follows. Let 
\begin{align*}
 \cM_0:=\left\{ 
 u\in H^1(\Omega)\::\: 
 \Big|
 |\{u>0\}|-|\{u<0\}|
 \Big|\leq |\{u=0\}|
 \right\}.
\end{align*}
Note that $\cM_0$ contains all solutions of \eqref{P0}, since $0=\int_\Omega \sign(u_0)=|\{u_0>0\}|-|\{u_0<0\}|$. Then the least-energy solution $u_0$ is the minimizer of
\begin{align*}
L_0=\inf_{\cM_0}I_0,\qquad  I_0(u)=\frac{1}{2}\int_\Omega |\nabla u|^2-\int_\Omega |u|
\end{align*}
It is known that $L_0$ is achieved by a function $u_0\in\cM_0$, that $u_0\in W^{2,q}_{loc}(\Omega)$ for all $q>0$, and that $u_0$ satisfies~\eqref{P0} pointwisely a.e. in $\Omega$, see \cite[Theorem 5.3]{PW15}.  Recall that the least-energy solution $u_p$ of~\eqref{eq:LENSNeumann} for $p\neq0$ close to 0 is the minimizer of 
\begin{align*}
 L_p=\inf_{\cM_p} I_p,\qquad I_p(u)=\frac{\|\nabla u\|_2^2}{2}-\frac{\|u\|_{p+1}^{p+1}}{p+1},\qquad 
 \cM_p=\{u\in H^1(\Omega)\::\: \int_\Omega |u|^{p-1}u=0\},
\end{align*}
which is achieved by \cite[Theorem 1.1]{PW15} or by Lemma \ref{lemma:D}.

\begin{theo}\label{thm:p0}
 Let $p_n\searrow 0$ as $n\to\infty$ and let $u_{p_n}$ be a l.e.n.s. of \eqref{eq:LENSNeumann}.  Then $L_{p_n}\to L_0$ as $n\to\infty$ and, up to a subsequence, $u_{p_n}\to u_0$ in $C^{1,\alpha}(\Omega)$ for any $\alpha\in(0,1)$ as $n\to\infty$, where $u_0$ is a least-energy solution of \eqref{P0}.
\end{theo}
\begin{proof}
 Observe that $(u_{p_n})$ is a bounded sequence in $H^1(\Omega)$; indeed, since $\|u_{p_n}\|_{p_n+1}=\min\limits_{c\in\R}\|u_{p_n}+c\|_{p_n+1}$, by Lemma \ref{eq:auxlemmaPW}, and $I_{p_n}(u_{p_n})=(\frac{1}{2}-\frac{1}{p_n+1})\|u_{p_n}\|_{p_n+1}^{p_n+1}<0$, we deduce that
 \begin{align*}
  \|\nabla u_{p_n}\|_2^2 
  &<\frac{2}{p_n+1}\|u_{p_n}\|^{p_n+1}_{p_n+1}
  =\frac{2}{p_n+1}\min\limits_{c\in\R}\|u_{p_n}+c\|^{p_n+1}_{p_n+1}\\
  &\leq 2|\Omega|^{\frac{1-p_n}{2}}\min\limits_{c\in\R}\|u_{p_n}+c\|^{p_n+1}_{2}
  \leq 2|\Omega|^{\frac{1-p_n}{2}}\mu_1^{\frac{p_n+1}{2}}\|\nabla u_{p_n}\|^{p_n+1}_{2},
 \end{align*}
and therefore $u_{p_n}$ is bounded in $H_1(\Omega)$.  Then there is $\widehat u\in H^1(\Omega)$ such that $u_{p_n}\rightharpoonup \widehat u$ in $H^1(\Omega)$ and $u_{p_n}\to \widehat u$ in $L^q(\Omega)$ for all $1\leq q<2^*$ as $n\to\infty$. In particular, by \cite[Lemma A.1]{W96} and dominated convergence,
\begin{align*}
 0 = \lim_{n\to\infty}\int_\Omega |u_{p_n}|^{p_n-1}u_{p_n} = \int_\Omega \sign(\widehat u)=|\{u>0\}|-|\{u<0\}|.
\end{align*}
In particular, $\widehat u\in \cM_0$.  But then
\begin{align}
 \liminf_{n\to\infty} L_{p_n} = \liminf_{n\to\infty}\left(
 \frac{\|\nabla u_{p_n}\|_2^2}{2}-\frac{\|u_{p_n}\|_{p_n+1}^{p_n+1}}{p_n+1}\right)
 \geq \left(
 \frac{\|\nabla \widehat u\|_2^2}{2}-\|\widehat u\|_{1}\right)\geq L_0.\label{l01}
\end{align}
On the other hand, by Lemma \ref{eq:auxlemmaPW}, there is a bounded sequence $(c_{n})\subset\R$ such that $u_0+c_{n}\in\cM_{p_n}$. Up to a subsequence, there is $\widehat c\in \R$ such that $c_n\to\widehat c$ as $n\to\infty$, and then
\begin{align}
 \limsup_{n\to\infty}L_{p_n}\leq \frac{\|\nabla u_0\|_2^2}{2}
 -\limsup_{n\to\infty}\frac{\|u_0+c_{n}\|_{{p_n}+1}^{{p_n}+1}}{{p_n}+1}
 &\leq \frac{\|\nabla u_0\|_2^2}{2}-\|u_0+\widehat c\|_{1}\notag\\
 &\leq \frac{\|\nabla u_0\|_2^2}{2}-\|u_0\|_{1} = L_0,\label{l02}
\end{align}
where we used that $\|u_0\|_{1}=\min_{c\in \R}\|u_0+c\|_{1}$, by Lemma \ref{eq:auxlemmaPW}.  But then~\eqref{l01} and~\eqref{l02} imply that $\lim_{n\to\infty}L_{p_n}=L_0$.
\medskip
Next, using the lower-semicontinuity of the gradient norm,
\begin{align}\label{lpl0}
 L_0=\liminf_{n\to\infty}L_{p_n}=\liminf_{n\to\infty}I_{p_n}(u_{p_n})\geq I_0(\widehat u)\geq L_0,
\end{align}
and therefore $\widehat u$ is a least-energy solution of~\eqref{P0}. Moreover, from~\eqref{lpl0}, we can deduce that $\|\nabla u_{p_n}\|_2^2\to \|\nabla \widehat u\|_2^2$ and therefore $u_{p_n}\to \widehat u$ in $H^1(\Omega)$.  Finally, by \cite[Lemma A.1]{W96}, there is $U\in L^1(\Omega)$ such that $|u_{p_n}|\leq |U|$ and, since $u_{p_n}$ is a $C^{2}(\Omega)$ solution of~\eqref{eq:LENSNeumann}, we have that
\begin{align*}
\int_{\Omega} |\Delta u_{p_n}|^{q} = \int_{\Omega}|u_{p_n}|^{p_nq} 
\leq  \int_{\Omega}(1+|U|)^{p_nq} \leq \int_{\Omega}(1+|U|)
\end{align*}
for all $n$ sufficiently large.  But then $(u_{p_n})$ is bounded in $W^{2,q}(\Omega)$ for any $q \geq 1$, and then, by Sobolev embeddings, $(u_{p_n})\subset C^{1,\alpha}(\Omega)$ for any $\alpha\in(0,1)$.  Since $C^{1,\alpha}(\Omega)$ is compactly embedded in $C^{1,\alpha'}(\Omega)$ for every $0<\alpha'<\alpha<1$, we have that $u_{p_n}\to \widehat u$ in $C^{1,\alpha'}(\Omega)$ for every $\alpha'\in(0,1)$, as claimed.
\end{proof}

Parts 2 to 4 of Theorem \ref{th:NeumannLENS} together with the continuity of the map $p\mapsto L_p$ for $p\in (0,\frac{N+2}{N-2}]\backslash\{1\}$ were proved in the previous section. Now we finish the proof by showing Part 1 and the continuity of $p\mapsto L_p$ at $p=0$.

\begin{proof}[Proof of Theorem \ref{th:NeumannLENS}-1] Item 1 and the continuity of $p\mapsto L_p$ at $p=0$ are a direct consequence of Theorem \ref{thm:p0}.
\end{proof}

\section{Symmetry, monotonicity, and symmetry-breaking}\label{symm:sec}
Let $N\geq 4$,
\begin{align*}
p_0:=2^*-1=\frac{N+2}{N-2},\qquad \cI:=(0,p_0],
\end{align*}
 and, for $p\in\cI$, let $u_p$ denote a least-energy solution of~\eqref{eq:LENSNeumann} with $\Omega$ being a ball or an annulus. In \cite[Corollary 1.4]{ST17}  it is shown that $u_p$ is foliated Schwarz symmetric and not radially symmetric if $0<p<p_0$ (see also \cite[Theorem 1.1]{PW15}). For $p=0$, in \cite[Theorem 1.2]{PW15} it is shown that $u_0$ is foliated Schwarz symmetric if $\Omega$ is a ball or an annulus, being not radial if $\Omega$ is a ball (up to our knowledge, the symmetry breaking on the annulus is an open problem).  In this section we extend these results to the critical case $p=p_0$, using the variational characterization~\eqref{D} and polarizations,  and we also show some new monotonicity results in the supercritical case if $\Omega$ is an annulus. 

\begin{remark}\label{rem:Pohozaev}
In the critical case $p=p_0:=2^*-1$, the fact that $u_{p_0}$ is \emph{not} radially symmetric can  be deduced from the Pohozaev identity if $\Omega$ is a ball, but in the case of annuli this necessarily requires another argument. Indeed, if $u_{p_0}$ were radial, by the proof of \cite[Theorem 2.1]{C10}, we would have that
\[
\left(\frac{r^N}{2N}(u_{p_0}'(r))^2+\frac{N-2}{2N}r^{N-1} u_{p_0}'(r)u_{p_0}(r)+\frac{N-2}{2N^2}r^N |u_{p_0}(r)|^\frac{2N}{N-2}\right)'=0,\ r=|x|.
\]
Thus, if $\Omega$ is a ball ($r\in [0,1)$), integrating between the origin (where $u_{p_0}'(0)=0$) and the first zero $r_0$ of $u_{p_0}$, we obtain that $\frac{r_0^N}{2N}(u'(r_0))^2=0,$ a contradiction by Hopf's Lemma or by the unique continuation principle.

However, if $\Omega$ is an annulus ($r\in (a,1)$ with $a>0$), then no contradiction arises, since integrating between $a$ (where $u_{p_0}'(a)=0$) and the first  zero $r_0$ yields
\[
\frac{N-2}{2N^2}a^N|u_{p_0}(a)|^\frac{2N}{N-2}=\frac{r_0^N}{2N}(u'(r_0))^2.
\]
\end{remark}

\subsection{Foliated Schwarz symmetry}\label{fss:ssec}

We introduce some standard notation. Let $N\geq 4$, $\Sn=\{x\in\mathbb R^N: |x|=1\}$ be the unit sphere, and fix $e\in \Sn$. We consider the halfspace $H(e):=\{x\in \mathbb R^N: x\cdot e>0\}$ and the half domain $\Omega(e):=\{x\in \Omega: x\cdot e>0\},$ where $\Omega$ is either a ball or an annulus.  

The composition of a function $w:\overline{\Omega}\to\R$ with a reflection with respect to $\partial H(e)$ is denoted by $w_e$, that is,
\begin{align*}
 w_e: \overline{\Omega}\to\R\qquad \text{ is given by }\qquad w_e(x):=w(x-2(x\cdot e)e).
\end{align*}
 
The \emph{polarization} $u^H$ of $u:\overline{\Omega}\to \R$ with respect to a hyperplane $H=H(e)$ is given by
\begin{align*}
 u^{H}:=\begin{cases} 
      \max\{u,u_e\} & \text{ in }\overline{\Omega(e)},\\
      \min\{u,u_e\} & \text{ in }\overline{\Omega}\ \backslash \, \overline{\Omega(e)}.
 \end{cases}
\end{align*}

We say that $u\in C(\overline{\Omega})$ is \textit{foliated Schwarz symmetric with respect to some unit vector $e^*\in \Sn$} if $u$ is axially symmetric with respect to the axis $\mathbb R e^*$ and nonincreasing in the polar angle $\theta:= \operatorname{arccos}(\frac{x}{|x|}\cdot e^*)\in [0,\pi].$ We use the following characterization of foliated Schwarz symmetry given in \cite{B03} (see also \cite{SW12} and \cite{W10}).
\begin{lemma}[Particular case of Proposition 3.2 in \cite{SW12}]\label{l:char}
There is $e^*\in\Sn$ such that $u\in C(\overline{\Omega})$ is foliated Schwarz symmetric with respect to $e^*$ if and only if for each $e\in \Sn$ either 
 \begin{align}\label{char}
u\geq u_e\quad \text{ in }\Omega(e)\qquad \text{ or }\qquad u\leq u_e\quad \text{ in }\Omega(e).
 \end{align}
\end{lemma}
We are ready to show the symmetry result. For $p<p_0$, this is a consequence of \cite{ST17}, which deals with the more general situation of Lane-Emden systems in the subcritical regime. However, the variational characterization therein is not the one associated to $D_p$ given in \eqref{D} and, for that reason, the proof is divided into sublinear and superlinear cases (see Theorems 4.8 and 4.9 in \cite{ST17}). Here we present a proof for the new case $p=p_0$. Since the argument also gives a unified approach for all $p\in \mathcal{I}$, we prove symmetry for all those ranges of $p$, for completeness.

\begin{theo}\label{th:sym}
 Let $\Omega$ be either a ball or an annulus centered at the origin of $\R^N$, take $p\in (0,2^*-1]\backslash\{1\}$, and let $f\in X_{p}$ be a maximizer of $D_{p}$ defined in~\eqref{D} and $u:=D_{p}^{\frac{1}{p-1}}|f|^{\frac{1}{p}-1}f$. Then there is $e^*\in\Sn$ such that $u$ and $f$ are foliated Schwarz symmetric with respect to $e^*$.
\end{theo}
\begin{proof}
Let $f$ and $u$ be as in the statement and fix a hyperplane $H=H(e)$ for some $|e|=1$.  By Proposition~\ref{prop:regularity} and Lemma~\ref{lemma:D}, $u\in C^{2,\alpha}(\overline{\Omega})$ for all $\alpha\in(0,1)$, $u$ solves~\eqref{eq:LENSNeumann} pointwisely, and $D_{p}^{-\frac{p}{p-1}}f=-\Delta u\in C^{0,\alpha}(\overline{\Omega})$; thus $f^H\in L^\infty(\Omega)$ and $\widetilde u:=D_{p}^{-\frac{p}{p-1}}K_{p} (f^H)\in W^{2,p}\cap C^{1,\alpha}$ for every $p\geq 1$ and $\alpha\in (0,1)$, by Lemma~\ref{l:reg} and Sobolev embeddings. 

Let $U:=u_e+u-\widetilde u - \widetilde u_e$, then using the definition of $f^H$ we have that 
\begin{align*}
-\Delta U = D_{p}^{-\frac{p}{p-1}}(f-f^H-(f^H)_e+f_e)=0\quad \text{ in }\Omega, \qquad \partial_\nu U=0\quad \text{ on } \partial \Omega. 
\end{align*}
Testing this equation with $U$ and integrating by parts we obtain that $U=k$ in $\Omega$ for some $k\in\R$. Then 
\begin{align}\label{reflections}
u_e+u=\widetilde u_e+\widetilde u+k\qquad \text{ in }\Omega.
\end{align}
Let 
\begin{align}\label{Gammas}
 \Gamma_1:=\{x\in\partial{\Omega(e)}\::\: x\cdot e=0\},\qquad \Gamma_2:=\{x\in\partial{\Omega(e)}\::\: x\cdot e>0\},
\end{align}
$w_1:=\widetilde u-u+k/2$, and $w_2:=\widetilde u-u_e+k/2$. Since $u=u^e$ and $\widetilde u=\widetilde u_e$ on $\Gamma_1$, we have that 
$w_1=w_2=0$ on $\Gamma_1$, by~\eqref{reflections}, and $\partial_\nu w_1=\partial_\nu w_2=0$ on $\Gamma_2$. Furthermore,
\begin{align*}
-\Delta w_1 =f^H - f\geq 0\quad \text{ in }\Omega(e)\qquad \text{ and }\qquad -\Delta w_2 =f^H - f_e\geq 0\quad \text{ in }\Omega(e),
\end{align*}
which implies by the maximum principle and Hopf's Lemma that $w_1 \geq 0$ and $w_2 \geq  0 $ in $\Omega(e)$.

By Lemma~\ref{lemma:D}, we know that $u=D_{p}^{-\frac{p}{p-1}}K_{p} f$.  Therefore, using that $\widetilde u_e=u_e+u-\widetilde u-k$ (by~\eqref{reflections}), $f^H_e=f_e+f- f^H$ (by definition of $f^H$), and that $\int_\Omega f = 0$, we obtain that
\begin{align}
D_{p}^{-\frac{p}{p-1}}\int_\Omega f K f- f^H K f^H\ dx&=\int_\Omega fu - f^H \widetilde u
 =\int_{\Omega(e)} fu + f_eu_e - f^H \widetilde u - (f^H)_e \widetilde u_e\, dx \nonumber\\
 &=\int_{\Omega(e)} fu + f_eu_e - f^H \widetilde u - (f_e+f- f^H)(u_e+u-\widetilde u-k)\, dx\nonumber\\
 &=\int_{\Omega(e)} (f_e-f^H)w_1 + (f-f^H)w_2 + \frac{k}{2}(f_e+f)\, dx \leq 0.\label{ineq}
\end{align}

To show that $u$ is foliated Schwarz symmetric with respect to some $e^*\in\Sn$, we use Lemma~\ref{l:char}. Assume, by contradiction, that~\eqref{char} does not hold. Then, without loss of generality, there are $e\in \Sn$ and the corresponding halfspace $H=H(e)$ such that
$u\neq u^H$ in $\Omega(e)$ and $u_e\neq u^H$ in $\Omega(e)$. Since $t\mapsto h(t):=|t|^s t$ is a strictly monotone increasing function in $\R$ for $s>-1$, this implies that
\begin{align*}
f = h(u)\neq h(u)^H = f^H  \qquad \text{ and }\qquad  f_e\neq f^H\quad\text{ in }\Omega(e).
\end{align*}
As a consequence, $w_1>0$ and $0\neq f-f^H\leq 0$ in $B(e)$ and $w_2>0$ in $B(e)$; but then,~\eqref{ineq} implies that 
\begin{align*}
\int_\Omega f K f < \int_\Omega f^H K f^H,
\end{align*}
Since $\|f^H\|_\frac{p+1}{p}=\|f\|_\frac{p+1}{p}$ (polarization is a rearrangement), this contradicts the maximality of $f$.  Therefore~\eqref{char} holds and the theorem follows from Lemma~\ref{l:char}.
\end{proof}

\subsection{Monotonicity of radial l.e.n.s. in annuli}
The main purpose of this subsection is to show existence and monotonicity of least-energy radial solutions of~\eqref{eq:LENSNeumann} in the critical and supercritical case $p\geq 2^*-1$ in annuli. We use the tools developed in \cite{ST17}, however this time applied to the variational framework characterization~\eqref{D}. 

For $p>1$, let $L_{rad}^{\frac{p+1}{p}}(\Omega)$ be the subspace of radially symmetric functions in $L^{\frac{p+1}{p}}(\Omega)$,
\begin{align}
\Omega&:=B_1(0)\backslash B_a(0)\quad \text{ for some }a\in(0,1),\label{Om}\\
 X_{p,rad}&:=\{f\in L_{rad}^\frac{p+1}{p}(\Omega): \int_\Omega f=0 \},\notag\\
 D_{p,rad}&:=\sup\left\{\int_\Omega fK f:\ f\in X_{p,rad},\ \|f\|_{\frac{p+1}{p}}=1\right\}.\notag
\end{align} 

If $f\in X_{p,rad}$ achieves $D_{p,rad}$ and $u:=D_{p}^{\frac{1}{p-1}}|f|^{\frac{1}{p}-1}f$, then we say that $u$ is a \emph{least-energy radial solution} of~\eqref{eq:LENSNeumann}, since, arguing as in the nonradial case, we have that $u$ achieves
\begin{align}
L_{p,rad}:=\inf \{I_{p}(w):\ w\in H^1_{rad}(\Omega)\setminus \{0\} \text{ is a weak solution of~\eqref{eq:LENSNeumann}}\}.
\end{align}
In the following we do a slight abuse of notation and use $w(|x|)=w(x)$ for a radial function $w$. We use $L^\infty_{rad}(\Omega)$ and $C_{rad}(\overline \Omega)$ to denote the subspace of radial functions in $L^\infty(\Omega)$ and $C(\overline \Omega)$, respectively.  Let $\Omega$ and $a$ be as in~\eqref{Om} and let
\begin{align*}
&{\cal I}:L_{rad}^\infty(\Omega)\to C_{rad}(\overline{\Omega}),\qquad {\cal I}h(x):=\int_{\{a{\leq}|y|{\leq}|x|\}} h(y)\ dy
=N\omega_N\int_a^{|x|}h(\rho)\rho^{N-1}\ d\rho\\
&\mF: C_{rad}(\overline{\Omega})\to L_{rad}^\infty(\Omega),\qquad \mF h:=(\chi_{\{{\cal I}h>0\}}-\chi_{\{{\cal I}h\leq 0\}})\, h.
\end{align*}
For $h\in C_{rad}(\overline{\Omega})$, the $\divideontimes$-transformation is given by 
\begin{align*}
h^{\divideontimes}\in L^\infty_{{rad}}(\Omega),\qquad h^{\divideontimes}(x) := ({\mathfrak F} h)^\#(\omega_N |x|^N-\omega_N \delta^N),
\end{align*}
where $\omega_N= |B_1|$ is the volume of the unitary ball in $\R^N$ and $\#$ is the decreasing rearrangement given by 
\begin{align*}
h^\#:[0,|\Omega|]\to\R,\qquad h^{\#}(0):=\text{ess sup}_U h\quad h^\#(s):=\inf\{t\in\R\::\: |\{h>t\}|<s\},\ s>0.
\end{align*}
For more details and comments regarding the definition of the \emph{flip-\&-rearrange transformation} $\divideontimes$ we refer to \cite[Section 3.2]{ST17}. We use the following result, which is a particular case of \cite[Theorem 1.3]{ST17} (for $f=g$, $p=q$) combined with \cite[Proposition 3.4]{ST17}.

\begin{theo}\label{thm:trans} Let $p>0$, $\Omega$ be as in~\eqref{Om}, and $f:\overline{\Omega}\to \R$ be a continuous and radially symmetric function with $\int_\Omega f = 0$. Then $f^\divideontimes\in X_{p,rad}$,
\begin{align}\label{in}
\|f^\divideontimes\|_\frac{p+1}{p}=\|f\|_\frac{p+1}{p}\quad \text{ and }\quad  \int_\Omega fK f\leq \int_\Omega f^\divideontimes K f^\divideontimes.
\end{align}
Furthermore, if $f$ is nontrivial and~\eqref{in} holds with equality, then $f$ is monotone in the radial variable. Moreover, $Kf$ is radially symmetric and strictly monotone in the radial variable.
\end{theo}

We are ready to show our monotonicity result. 

\begin{theo} \label{thm:mono}
Let $\Omega$ be as in~\eqref{Om} and let $p>1$. The set of least-energy radial solutions of~\eqref{eq:LENSNeumann} is nonempty. Moreover, if $u$ is a least-energy radial solution, then $u\in C^{2,\alpha}(\overline{\Omega})$ is a classical radial solution of~\eqref{eq:LENSNeumann} and $u$ is strictly monotone in the radial variable.
\end{theo}
\begin{proof}
The existence of least-energy radial solutions for $p>1$ follows similarly as in Lemma~\ref{lemma:D}. Indeed, by letting $\psi_{1}$ be the first non-constant \emph{radial} eigenfunction, we have that $\psi_{1}\in X_{p,rad}$ and $D_{p,rad}>0$.  Recall now that, for $t>1$, the embedding 
\begin{align}\label{ce}
W_{rad}^{2,t}(\Omega)\subset W_{rad}^{1,t}(\Omega) \hookrightarrow C^0(\overline \Omega)
\end{align}
is compact  (this is well known and can be proved using radial variables and reducing the problem to the 1-dimensional case $W^{1,t}(a,1)$). Therefore, the operator $K$ is compact from $L_{rad}^\frac{p+1}{p}(\Omega)$ to $C^0(\overline \Omega)$, and in particular to $L_{rad}^{p+1}(\Omega)$. Reasoning now as in Lemma~\ref{lemma:D}-(ii), if follows that $D_{p,rad}$ is achieved at some $f\in X_{p,rad}$ such that $\| f\|_{\frac{{p}+1}{{p}}}=1$, and $u=D_{p,rad}^{-\frac{p}{p-1}}K_{p} f$ is a radial least-energy solution of~\eqref{eq:LENSNeumann}.

Let $f\in X_{p,rad}$ be a maximizer for $D_{p,rad}$. Arguing as in Proposition~\ref{prop:regularity} (using \eqref{ce}) we have that $f$ is continuous up to the boundary $\partial \Omega$. By Theorem~\ref{thm:trans} and the maximality of $f$ we have that $f$ is strictly monotone in the radial variable and therefore $u=D_{p,rad}^{-\frac{p}{p-1}}K_{p} f=D_{p,rad}^{\frac{1}{p-1}}|f|^{\frac{1}{p}-1}f$ is also strictly monotone in the radial variable.
\end{proof}

\subsection{Symmetry breaking in annuli}

We use the following two lemmas. 

\begin{lemma}\label{lemma1}
Let $f\in X_{p,rad}$ be a maximizer for~\eqref{D} with $p\in\cI_{rad}$, $p>1$, and $u:=D_{p}^{-\frac{p}{p-1}}K_{p} f$.  Then $u\in C^{3,\alpha}(\overline{\Omega})$. Moreover, if $u$ is increasing in the radial variable and $\bar f:= p |u|^{p-1} u_{x_1},$ then 
\begin{align}\label{ap:pos:1}
K_{p} \bar f\geq 0\quad \text{ in }\Omega(e_1)
\end{align}
 and
\begin{align}\label{st:pos}
\infty> \int_{\Omega(e_1)} p|u|^{p-1}u_{x_1}(u_{x_1}-K_{p}\bar f )\, dx\geq 0.
\end{align}
\end{lemma}
\begin{proof}
Let $p\geq 2^*-1$. By Proposition~\ref{prop:regularity} we know that $u\in C^{2,\alpha}(\overline{\Omega})$ for every $\alpha\in(0,1)$, but then $-\Delta u = |u|^{p-1}u\in C^{1,\alpha}(\overline{\Omega})$ (because $p>1$) and, by elliptic regularity, $u\in C^{3,\alpha}(\overline{\Omega})$. Moreover, by Lemma~\ref{lemma:D} we know that $u$ achieves~\eqref{c:eps}. The claim now follows from \cite[Lemma 4.5]{ST17} (using $p=q$, observe that \cite[Lemma 4.5]{ST17} is stated in the subcritical regime, but the same proof holds for the critical and supercritical cases if the assumptions hold).
\end{proof}

\begin{lemma}[Lemma 4.6 from \cite{ST17}]\label{lemma2}
 Let $t>1$ and $h\in L^t(\Omega)\backslash\{0\}$ be an antisymmetric function in $\Omega$ with respect to $x_1$ and let $w^N:=K h\in W^{2,t}(\Omega)$, that is, $w^N$ is the unique strong solution of 
 \begin{align*}
 -\Delta w^N=h\quad \text{ in }\Omega,\qquad\partial_\nu w^N=0\quad \text{ on }\partial \Omega,\qquad \text{ and }\qquad \int_\Omega w^N =0.
 \end{align*}
Moreover, let $w^D\in W^{2,t}(\Omega)\cap C^1(\overline{\Omega})$ be a strong solution of $-\Delta w^D=h$ in $\Omega$ with $w^D=0$ on $\partial \Omega$.  If
$h\geq0$ and $w^N\geq 0$ in $\Omega(e_1)$ then $w^D<w^N$ in $\Omega(e_1)$.
\end{lemma}

\begin{theo}\label{nr:tm:intro}
 Let $N\geq 4$, $\Omega$ as in~\eqref{Om} and $f$ be a maximizer for~\eqref{D} for $p\in \cI_{rad}$, $p>1$, then $f$ is not radially symmetric.  In particular, least-energy solutions of~\eqref{eq:LENSNeumann} in annuli are not radially symmetric.
 \end{theo}
\begin{proof}
Let $\Omega\subset \R^N$ be as in~\eqref{Om}, $f$ as in the statement, $u:= D_{p}^{-\frac{p}{p-1}}K_{p} f$, and $\bar f:= p |u|^{p-1} u_{x_1}$.  We argue by contradiction. Assume that $f$ is radial. Then $u$ is also radially symmetric and, by Theorem~\ref{thm:mono}, $u$ is strictly monotone in the radial variable. Since $u\in C^{3,\alpha}(\overline{\Omega})$, by Lemma~\ref{lemma1}, we may interchange derivatives and $(-\Delta u)_{x_1}=-\Delta (u_{x_1})$ in $\Omega$. Thus $u_{x_1}\in C^{2,\alpha}(\overline{\Omega})$ is the unique solution of the Dirichlet problem
\begin{align*}
-\Delta u_{x_1} &= p |u|^{p-1} u_{x_1} \quad  \text{ in }\Omega\qquad\text{with}\quad u_{x_1}=0\quad \text{ on }\partial\Omega,
\end{align*}
where the boundary conditions follow from the fact that $u$ is radially symmetric and $\partial_\nu u=0$ on $\partial\Omega$. We may assume without loss of generality that $u$ is strictly \emph{increasing} in the radial variable (the other case follows similarly). Then $u_{x_1}$ is nonnegative in $\Omega(e_1)$ and, by Lemmas~\ref{lemma1} and~\ref{lemma2},
\begin{align*}
0&\leq \int_{\Omega(e_1)} p|u|^{p-1}u_{x_1}(u_{x_1}-K_{p}\bar f )\, dx <0,
 \end{align*}
a contradiction. Therefore $f$ cannot be radially symmetric and this concludes the proof.
\end{proof}

\subsection{Proof of the main results in the radial setting}\label{proofsradial:sec}

\begin{proof}[Proof of Proposition \ref{prop:partialsymmetry}] If $\Omega$ is an annulus and $p>1$, the result follows from Theorems \ref{th:sym}, \ref{nr:tm:intro}, and Lemmas \ref{lemma:DpLambdap}, \ref{lemma:D_p=1/Lambda}.  The other cases ($\Omega$ being a ball and $p\in(0,2^*-1)$ or $\Omega$ being an annulus and $p\in(0,1]$) were previously known and can be found in \cite{ST17} (for $p\neq 1$) and \cite{GiraoWeth} (for $p=1$).
\end{proof}

\begin{proof}[Proof of Theorem~\ref{th:NeumannLENS:radial}] 
We only argue the case $N\geq 3$, where $2^*<\infty$. The cases  $N=1,2$ follow with similar arguments. The existence of least-energy radial solutions for $p\in \cI_{rad}\backslash\{1\}$ follows from Theorem \ref{thm:mono} if $\Omega$ is an annulus and from \cite[Theorem 1.2]{ST17} if $\Omega$ is a ball (or one can also argue analogously as in Theorem \ref{thm:mono}). \textbf{Part 5} is a consequence of the Pohozaev inequality and the fact that every radial solution of \eqref{eq:LENSNeumann} is sign-changing (and therefore necessarily implies the existence of a Dirichlel radial solution in a radial subdomain). The convergence results in \textbf{Parts 1,2,3,} and \textbf{4} follow from analogous arguments as those used in Theorem \ref{th:NeumannLENS} using functional spaces with radially symmetric functions.  We omit the details.
\end{proof}

\begin{proof}[Proof of Proposition \ref{new:prop}]
 This follows from Theorem \ref{thm:mono}.
\end{proof}

\section{Numerical approximation of radial solutions}\label{num:sec}

The next figure exemplifies the different shapes of monotone radial solutions in annuli for $p\geq 0$. 
\begin{center}
\includegraphics[width=0.3\textwidth]{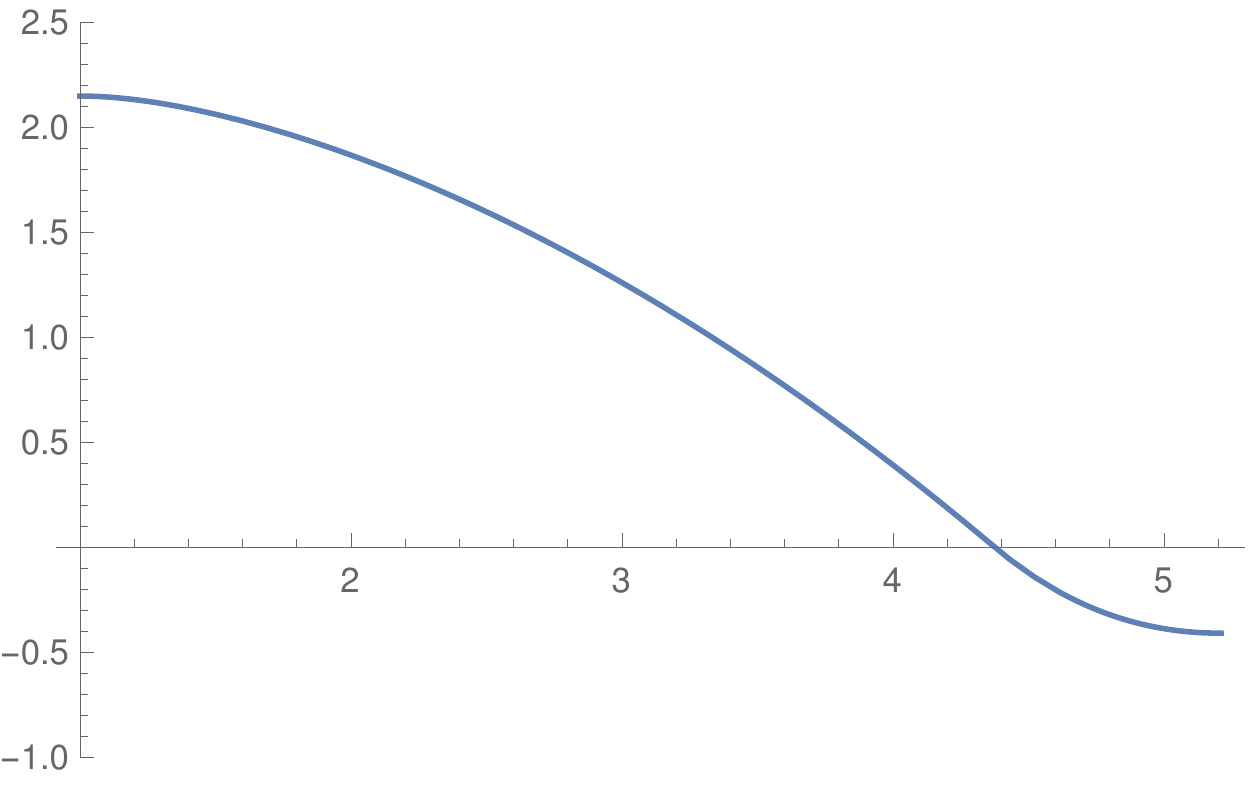}\ 
\includegraphics[width=0.3\textwidth]{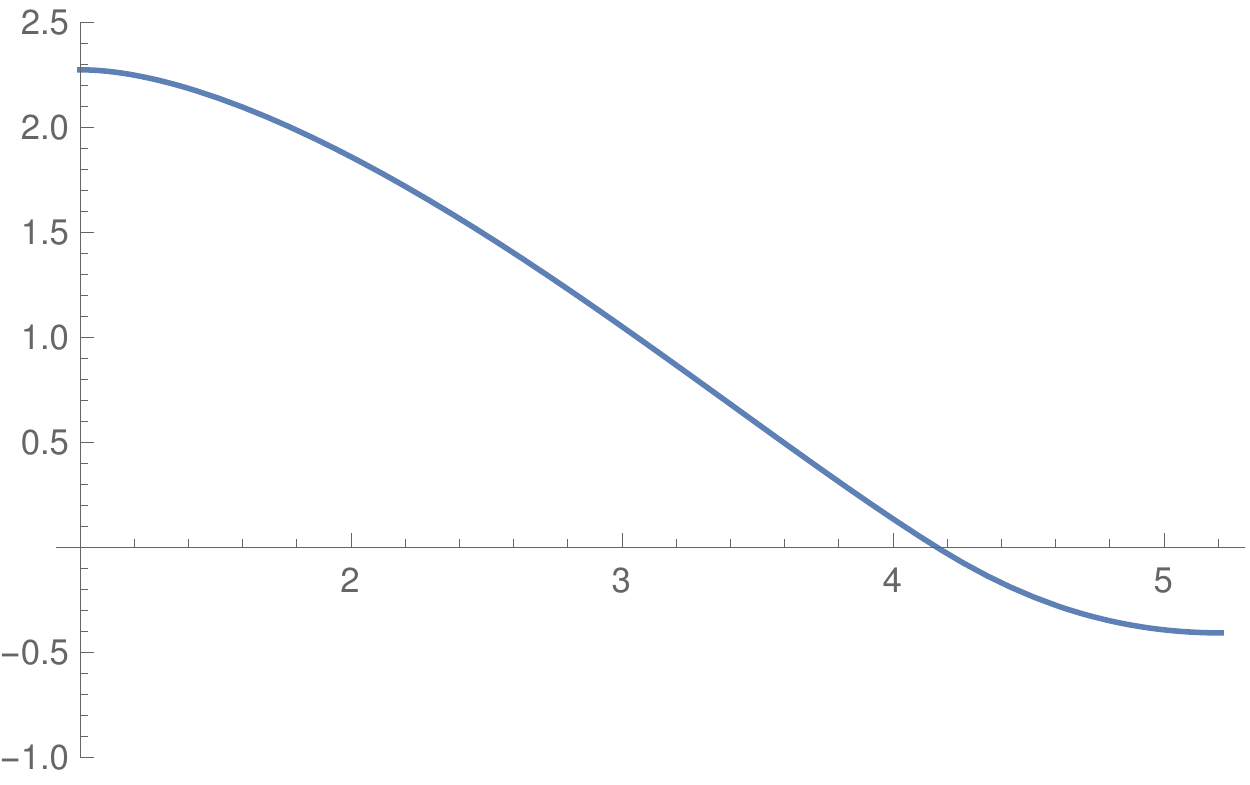}\ 
\includegraphics[width=0.3\textwidth]{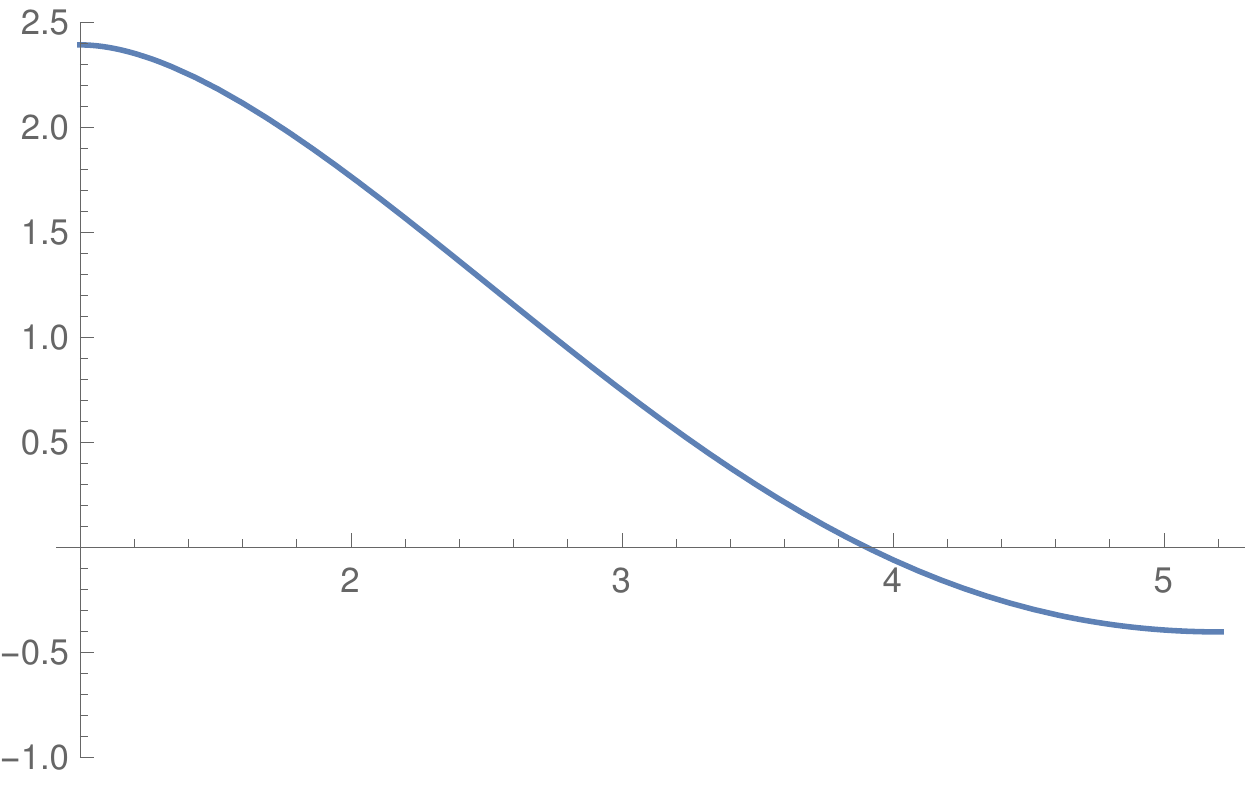}\ 
\includegraphics[width=0.3\textwidth]{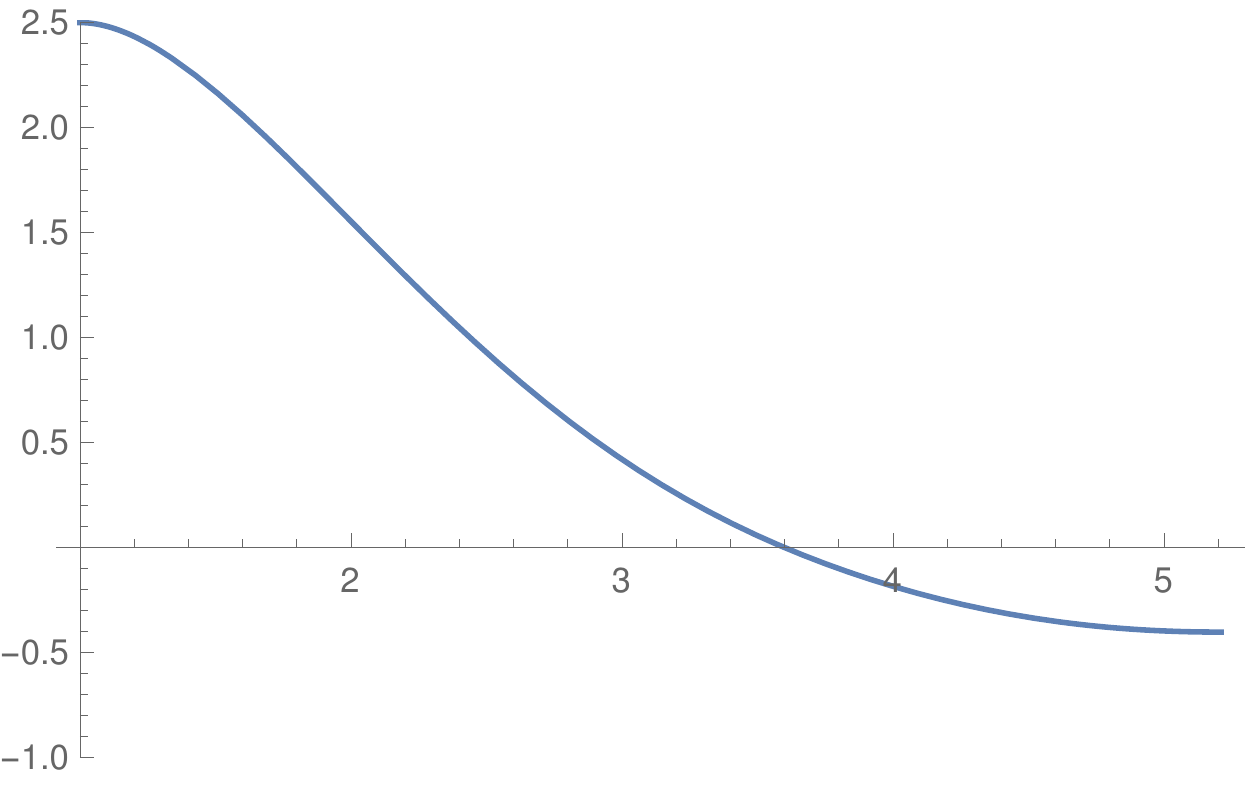}\ 
\includegraphics[width=0.3\textwidth]{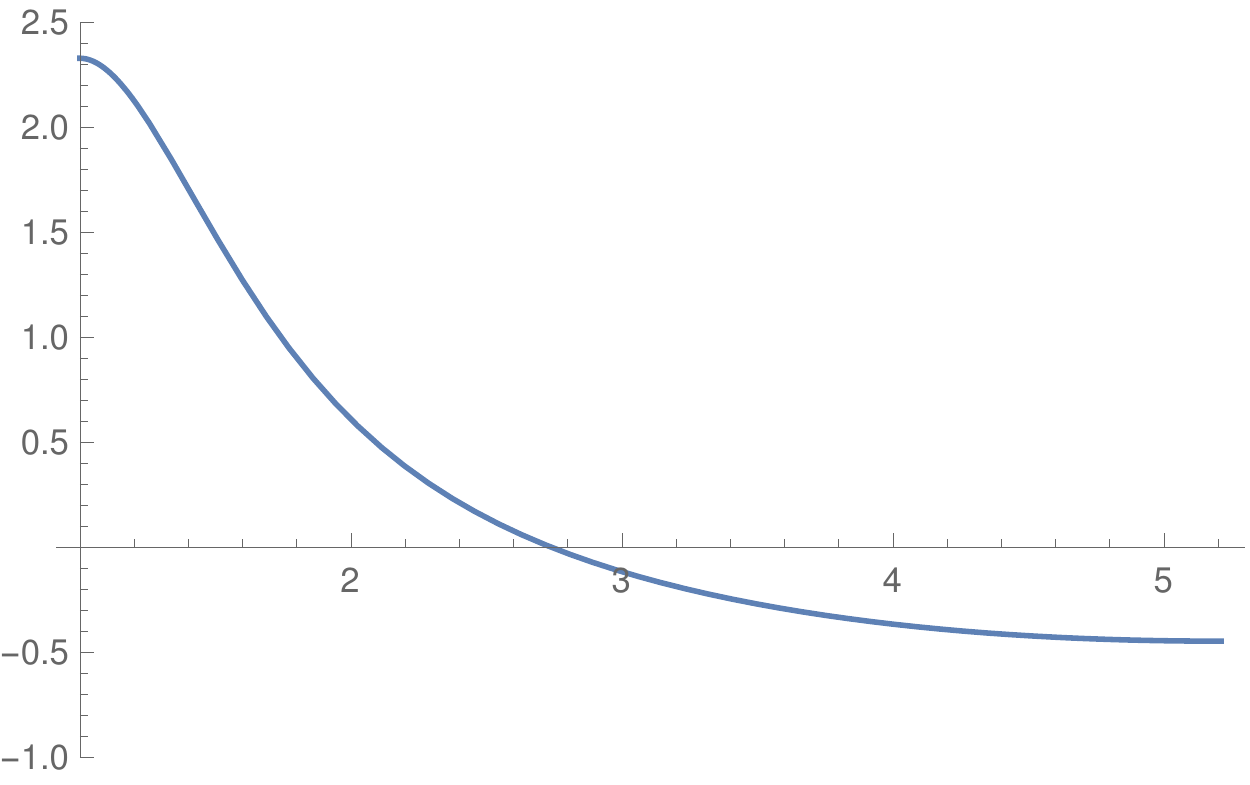}\ 
\includegraphics[width=0.3\textwidth]{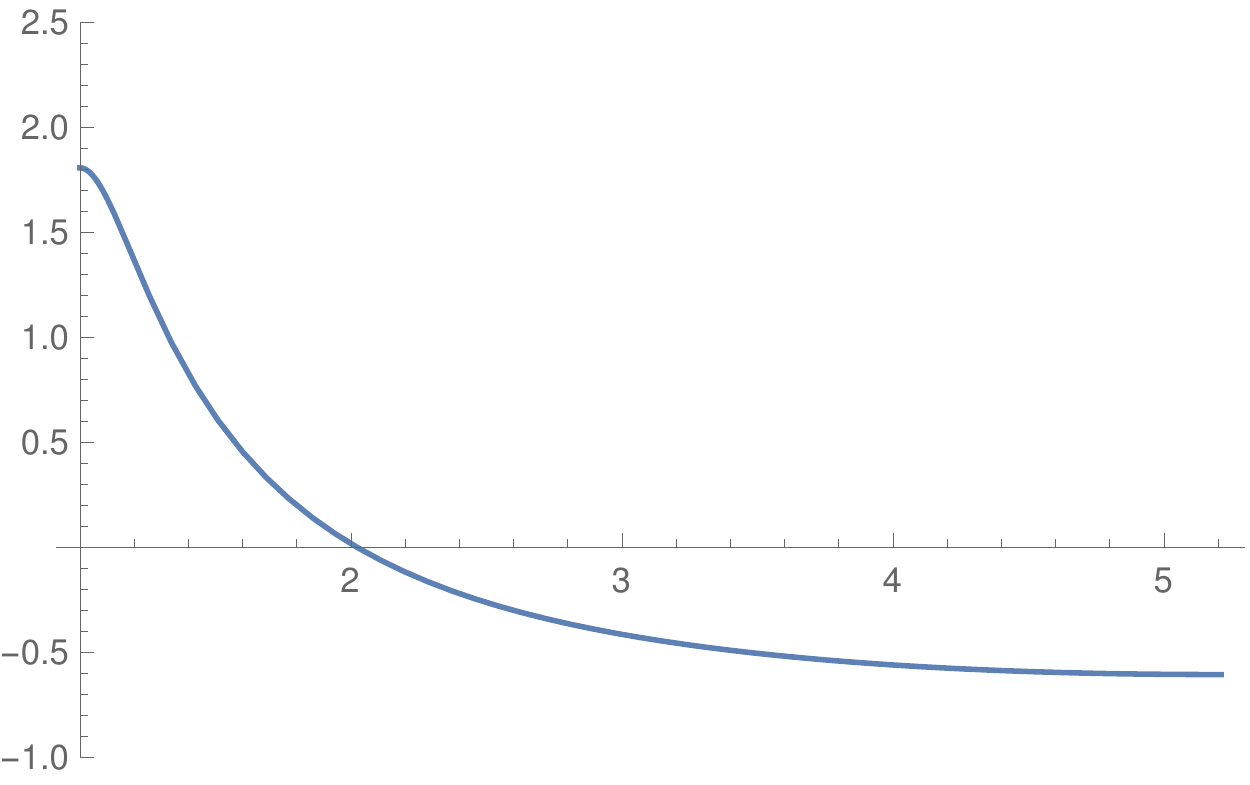}\ 
\includegraphics[width=0.3\textwidth]{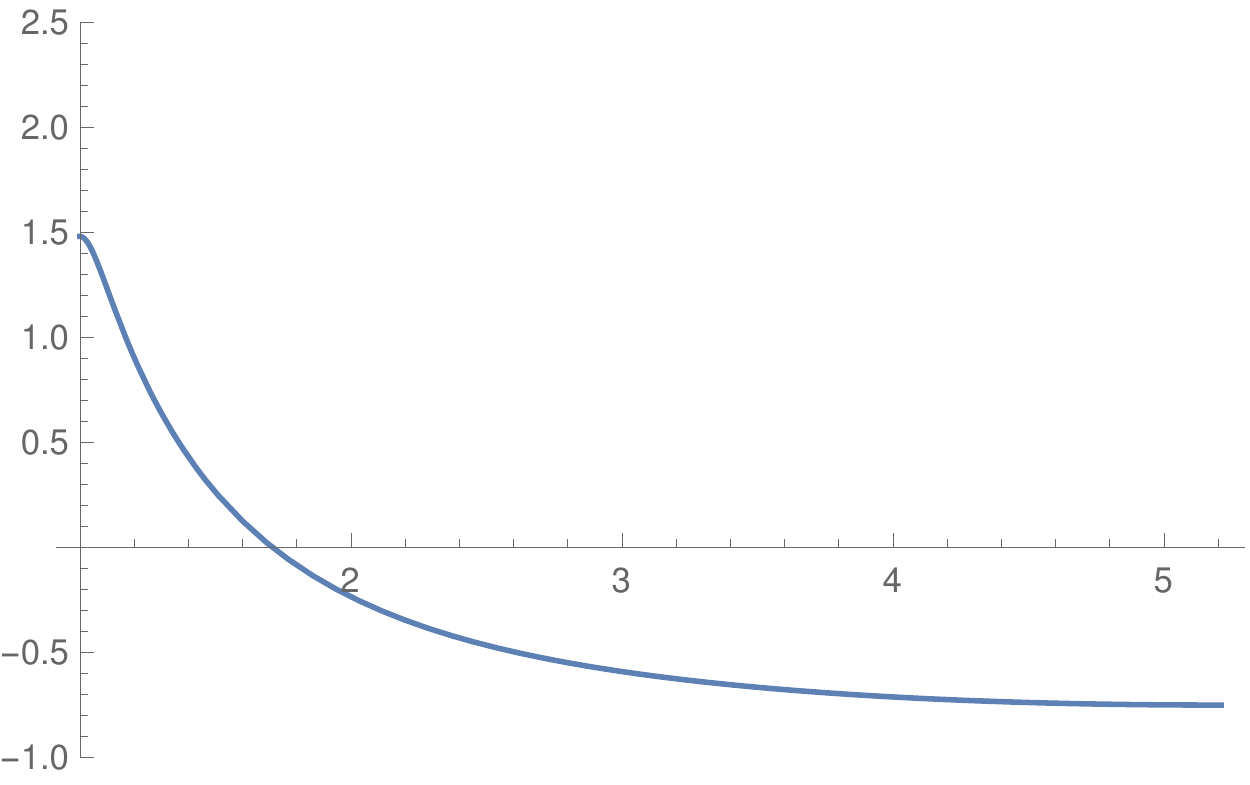}\ 
\includegraphics[width=0.3\textwidth]{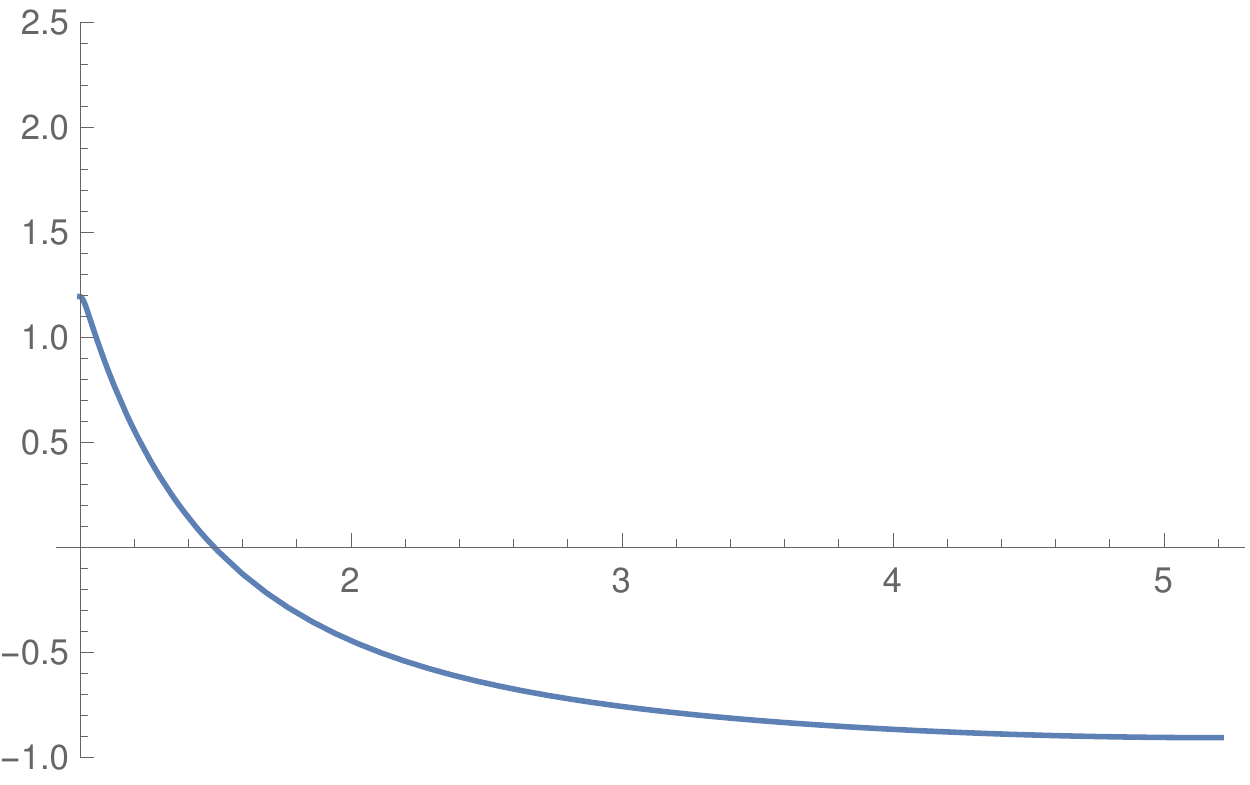}\ 
\includegraphics[width=0.3\textwidth]{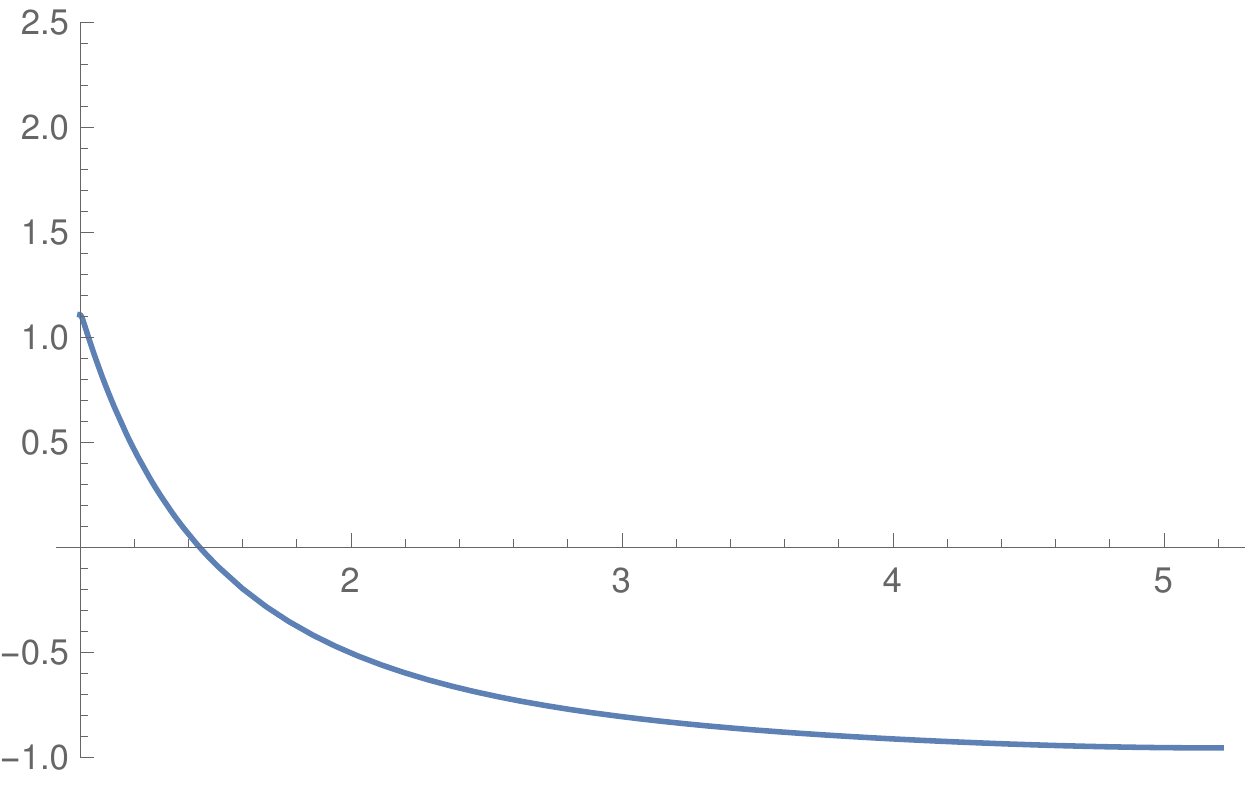}\ 
\captionof{figure}{From left to right and top to bottom: a numerical approximation of the radially symmetric solution of \eqref{eq:LENSNeumann} in  the annulus $A=\{x\in\R^4\::\: 1<|x|<5.21021\}$ for $p = 0,\ 0.5,\ 1,\ 1.5,\ 3,\ 6,\ 11,\ 31,\ 61.$  At $p=1$, the solution $u_1$ given by \eqref{u1rad} is depicted. The annulus $A$ is selected in such a way that $\mu_{1,rad}(A)\approx 1.$} \label{fig1}
 \end{center}

\medbreak

\noindent  \textbf{{Acknowledgments.}} H. Tavares was partially supported by the Portuguese government through FCT-Funda\c c\~ao para a Ci\^encia e a Tecnologia, I.P., under the projects UID/MAT/04459/2020 and PTDC/MAT-PUR/28686/2017.

\end{document}